\documentclass{amsart}

\usepackage{amssymb}
\usepackage{amsmath}
\usepackage{bbm, dsfont}
\usepackage{graphics}
\usepackage{mathrsfs}
\usepackage{color,cite,graphicx}
\definecolor{gray}{gray}{0.4}
\usepackage[all]{xy}
\usepackage{enumerate}

\title[Order four]{Symmetries of order four on K3 surfaces}

\author{Michela Artebani}
\address{Departamento de Matem\'atica, Universidad de Concepci\'on, Casilla 160-C, Concepci\'on, Chile}
\email{martebani@udec.cl}

\author{Alessandra Sarti}
\address{Laboratoire de Math\'ematiques et Applications, UMR CNRS 7348,
			Universit\'e de Poitiers, T\'el\'eport 2, Boulevard Marie et Pierre Curie,
			F-86962 FUTUROSCOPE CHASSENEUIL}
\email{sarti@math.univ-poitiers.fr}
\urladdr{http://www-math.sp2mi.univ-poitiers.fr/~sarti/}

\date{\today}

\newtheorem{lemma}{Lemma}
\newtheorem{pro}{Proposition}
\newtheorem{cor}{Corollary}
\newtheorem{theorem}{Theorem}[section]
\theoremstyle{definition}

\newtheorem{example}[theorem]{Example}

\newtheorem{remark}[theorem]{Remark}
\newcommand{\bprf}{{\it Proof}.~}
\DeclareMathOperator{\Pic}{Pic}

\DeclareMathOperator{\rank}{rk}
\DeclareMathOperator{\id}{id}

\DeclareMathOperator{\Fix}{Fix}

\DeclareMathOperator{\tr}{tr}
\DeclareMathOperator{\rk}{rk}

\newcommand{\IC}{\mathbb{C}}
\newcommand{\IQ}{\mathbb{Q}}
\newcommand{\IZ}{\mathbb{Z}}

\newcommand{\IR}{\mathbb{R}}
\newcommand{\IP}{\mathbb{P}}
\newcommand{\calO}{\mathcal{O}}

\newcommand{\map}{\rightarrow}


\thanks{The authors have been partially 
supported by Proyecto FONDECYT Regular 2009, N. 1090069
and by Proyecto FONDECYT Regular 2011, N. 1110249.}

\begin{document}

\begin{abstract}
In this paper we study automorphisms of order four on K3 surfaces. 
We give a classification of the non-symplectic ones when either
the square of the automorphism is symplectic, or
its fixed locus contains a curve of positive genus, or
its fixed locus contains at least a curve and all the curves fixed by its square are rational.   
We provide partial results in the other cases.
\end{abstract}

\maketitle

\section*{Introduction}
Let $X$ be a K3 surface over $\mathbb C$ with an order four automorphism.
Such automorphism acts on the one-dimensional vector space $H^{2,0}(X)$ of holomorphic two-forms of $X$ either as the identity,
minus the identity or as the multiplication by $\pm i$.
Accordingly, the automorphism is called symplectic, with symplectic square or purely non-symplectic. 
Symplectic automorphisms of finite order have been investigated by several authors, their fixed locus only contains isolated points and their action on the K3 lattice $H^2(X,\IZ)$ is known to be independent from the surface (cf.\cite{Nikulin1, Alice1, Alice2}).
Non-symplectic automorphisms have been studied in \cite{Nikulin1, MO} (see also \cite{zhang} for a survey on the topic) and the fixed locus has been identified if the order is prime in \cite{AS3,takiauto, ast}. 
In \cite{Taki} Taki classified order four non-symplectic automorphisms acting as the identity on the Picard lattice of the surface. Moreover, Sch\"u{tt} \cite{Schuett} studied the special case when the transcendental lattice of the surface has rank four.

This paper mainly deals with purely non-symplectic automorphisms $\sigma$ of order four under the assumption that their square is the identity on the Picard lattice. 
By the Torelli type theorem, this holds for the generic element of the family of K3 surfaces carrying 
an order four non-symplectic automorphism with a given action on its second cohomology group.
The fixed locus  $\Fix(\sigma)$  of such an automorphism $\sigma$ is the disjoint union of smooth curves and points. We give a complete classification of  $\Fix(\sigma)$ when either it contains a curve of positive genus or it contains a curve and all the curves fixed by $\sigma^2$ are rational. 
More precisely, we denote by $r$ and $l$ the ranks of the eigenspaces of  $\sigma^*$ in  $H^2(X,\IZ)$ with eigenvalues $1$ and $-1$, by $n$ and $k$ the number of isolated points and of smooth rational curves in $\Fix(\sigma)$ and by $2a$ the number of smooth curves fixed by $\sigma^2$ and interchanged by $\sigma$.
%we denote by $g$ the highest genus of a curve in $\Fix(\sigma)$, by $k$ the number of smooth rational curves in $\Fix(\sigma)$, by $2a$ the number of smooth curves fixed by $\sigma^2$ and interchanged by $\sigma$ and by $n$ the number of isolated fixed points. 
 We prove the following result.

\begin{theorem}\label{teoremone}
Let $\sigma$ be a purely non-symplectic order four automorphism on a K3 surface $X$ such that $(\sigma^*)^2$ acts identically on $\Pic(X)$. Then:
$$n=2\alpha+4,\quad r-l=4\alpha+2,$$
where $\alpha=\sum_{C\subset \Fix(\sigma)} 1-g(C)$.
Moreover the following hold.
\begin{itemize}
\item If $\Fix(\sigma)$ contains a curve of genus  $g=1$, then all the other fixed curves are rational and we have the following possibilities for $(r,k,a)$: 
$$(6,0,0),\ (7,0,0),\ (10,1,0),\ (8,0,1),\ (9,0,2),\ (10,0,3).$$
\item If $\Fix(\sigma)$ contains a curve of genus $g>1$, then all the other fixed curves are rational and we have the following possibilities for $(r,k,a,g)$: 
$$(1,0,0,3),\ (4,0,0,2),\ (2,0,1,3),\ (5,0,1,2),\ (6,0,2,2),$$ 
where all the cases occur. 
\item If $\Fix(\sigma)$ contains a curve of genus $g=0$ and all the curves fixed by $\sigma^2$ are rational, then we have the following possibilities for $(r,k,a)$: 
$$(10,1,0),\ (13,2,0),\ (11,1,1),\ (16,3,0),$$
$$(14,2,1),\ (12,1,2),\ (19,4,0),\ (13,1,3),$$ 
where all the cases occur. 
\end{itemize}
\end{theorem}

We prove Theorem \ref{teoremone} in several steps, mainly in  theorems \ref{g1}, \ref{propocurve} and \ref{rational}. We give examples showing the existence in examples \ref{wei}, \ref{quartic}, \ref{hyp}, \ref{ex0}.  In the cases not considered in the theorem, i.e. when either $\Fix(\sigma)$ only contains isolated points or $\sigma^2$ fixes a curve of positive genus which is not fixed by $\sigma$, 
we provide partial results and we give several examples. More precisely, we show the following theorem, where we use the same notation as before.
\begin{theorem}\label{teoremone2}
Let $\sigma$ be a purely non-symplectic order four automorphism on a K3 surface $X$ such that $(\sigma^*)^2$ acts identically on $\Pic(X)$. Then the following hold. 
\begin{itemize}
\item If $\sigma^*$ acts as the identity on $\Pic(X)$, then the fixed locus of $\sigma$ only contains smooth rational curves and points,
$\sigma^2$ fixes a curve of genus $g\geq 1$ and we have the following possibilities for 
$(r,k,g)$: 
$$(2,0,10),\ (2,0,9),\ (6,1,7),\ (6,1,6),\ (10,2,g), \ 3\leq g\leq 6,$$
$$(14,3,3),\ (14,3,2),\ (18,4,2),\ (18,4,1),$$ 
where all the cases occur.
\item If $\Fix(\sigma)$ is zero-dimensional, then it contains exactly $4$ points. If $l>0$, 
 then the possible invariants of $\Fix(\sigma^2)$ are given in Table \ref{alpha=0} 
 and we have $3\leq r\leq 11$, $a\leq 4$, $g\leq 8$, where $g$ is the highest
genus of a curve fixed by $\sigma^2$.
%then there are thirty possibilities for the fixed locus $\Fix(\sigma^2)$ with  $3\leq r\leq 11$, $0\leq a \leq 4$ and $0\leq \gamma\leq 8$, where $\gamma$ denotes the highest genus of a curve fixed by $\sigma^2$. 

\item If we are not in one of the previous cases or in one of the cases of Theorem \ref{teoremone}, i.e. $\Fix(\sigma)$ contains only $k>0$ rational curves and isolated fixed points, the highest genus $g$ of a curve fixed by $\sigma^2$ is positive  and $l>0$, then 
\begin{itemize}
\item if $g>1$ there are $63$ possible cases with $k\leq 3$, $g\leq 7$ and $a\leq 4$;
\item if $g=1$ we have the following possibilities for
$(r,k,a)$: 
$$(9,1,0),\ (12,2,0),\ (15,3,0),\ (10,1,1),\ (10,1,0).$$ 
\end{itemize}

\end{itemize}
\end{theorem}
We prove this result in theorems \ref{teotaki}, \ref{teopunti}, \ref{curvefissate} and \ref{othergenus1} and give examples in \ref{hirz}, \ref{quadric}, \ref{hirz1}, \ref{jaco}, \ref{wei2}, \ref{singu}, \ref{ultimo}.  

Finally, in  Proposition \ref{symp} we consider the case when the automorphism has symplectic square.
 
The study of non-symplectic automorphisms of order four  is
interesting in relation with the Borcea-Voisin construction of Calabi-Yau varieties 
and the investigation of Mirror-Symmetry (cf. \cite{Borcea, Voisin}). In fact Borcea and Voisin consider the product between a K3 surface with a non-symplectic automorphism of order 2,3,4 or 6, and an elliptic curve with an automorphism of the same order. A resolution of a quotient variety
is then a Calabi-Yau threefold. In \cite{Alice0} Garbagnati used non-symplectic automorphisms of order four to give examples of Calabi-Yau threefolds by means of this construction.

We now give a short description of the paper's sections. 
%Let $\sigma$ be a non-symplectic automorphism of order four on a K3 surface $X$.
In section 1 we give a general description of the fixed locus of $\sigma$. 
%In case it is purely non-symplectic, it is the disjoint union of $n$ points, $k$ smooth rational curves and possibly a smooth curve of genus $g$. 
By means of Lefschetz's formulas we provide relations between the invariants $n,k,g$ and the ranks of the eigenspaces of $\sigma^*$ on the vector space $H^2(X,\IC)$.
%If $\sigma$ has symplectic square, then we classify the fixed locus according to the rank of the fixed lattice. 

In section 2 we study elliptic fibrations $\pi:X\map \IP^1$ such that $\sigma$ preserves each fiber of $\pi$. In Corollary \ref{cor}  the configuration of the singular fibers, which are of Kodaira type ${\rm III, I_0^*}$ or ${\rm III^*}$,  is related to the structure of the fixed locus of $\sigma$.

In section 3 we assume that $\sigma$ fixes pointwisely an elliptic curve $E$. In Theorem \ref{g1} we describe the singular fibers of the elliptic fibration with fiber $E$ and the corresponding structure of the fixed locus of $\sigma$.

In section 4 and 5 we classify the case when the fixed locus of $\sigma$ either contains a curve of positive genus, or it contains a rational curve and $\sigma^2$ fixes only rational curves.
 
In section 6 we assume that $\sigma^*$ is the identity on the Picard lattice and
 we give an independent proof of \cite[Proposition 4.3]{Taki}.
 
 In section 7 we consider the case when $\sigma$ only fixes isolated points and we provide families of examples.
 
 In section 8, we study the case when $\sigma$ only fixes isolated points and rational curves, and $\sigma^2$ fixes a curve of positive genus.  \\

{\em Acknowledgements:} We warmly thank Bert van Geemen for useful discussions 
and Alice Garbagnati for several remarks and corrections.  
We also thank the anonimous referee for many comments and 
for pointing out a mistake in the first version. 

\section{The fixed locus}

Let $X$ be a K3 surface with a {\it non-symplectic} automorphism $\sigma$ of order four, 
i.e. such that the action of $\sigma^*$ on the vector space $H^{2,0}(X)=\IC \omega_X$ of holomorphic two-forms is not trivial. We call the automorphism {\it purely non-symplectic} if $\sigma^*\omega_X=\pm i \omega_X$, i.e. if its square is a non-symplectic involution. 
%Otherwise $\sigma^*\omega_X=- \omega_X$, so that $\sigma^2$ is a symplectic involution.

 We will denote by $r,l,m$ the rank of the eigenspace of $\sigma^*$ in $H^2(X,\IC)$ relative to the eigenvalues $1,-1$ and $i$ respectively. Moreover, let
$$
S(\sigma)=\{x\in H^2(X,\IZ)\,|\, \sigma^*(x)=x\},
$$
$$
S(\sigma^2)=\{x\in H^2(X,\IZ)\,|\, ({\sigma^2})^*(x)=x\},\quad T(\sigma^2)=S(\sigma^2)^{\perp}\cap H^2(X,\IZ).
$$ 
Observe that $r=\rank S(\sigma)$,  $r+l=\rank S(\sigma^2)$ and $2m=\rank T(\sigma^2)$. 
%We will also denote by $g_i$ the maximal genus of a curve in $\Fix(\sigma^i)$ and 
%by $k_i$ the number of smooth rational curves in $\Fix(\sigma^i)$ for $i=1,2$.

We start recalling the following result about non-symplectic involutions (see \cite[Theorem 4.2.2]{nikulinfactor}, \cite[Theorem 6.1]{K1} or \cite[Theorem 4.1]{ast}).
\begin{theorem}\label{inv}
Let $\tau$ be a non-symplectic involution on a K3 surface $X$. The fixed locus of $\tau$ is 
either empty, the disjoint union of two elliptic curves or the disjoint union of a 
smooth curve of genus $\gamma\geq 0$ and 
$j$ smooth rational curves. 

Moreover, its fixed lattice $S(\tau)\subset \Pic(X)$ is a $2$-elementary lattice with determinant $2^d$ such 
that:
\begin{itemize}
\item $S(\tau)\cong U(2)\oplus E_8$ if the fixed locus of $\tau$ is empty;
\item  $S(\tau)\cong U\oplus E_8(2)$ if $\tau$ fixes two elliptic curves;
\item $2\gamma=22-\rk S(\tau)-d$ and $2j=\rank S(\tau)-d$ otherwise.
\end{itemize}

%Let $\sigma$ be a purely non-symplectic automorphism of order four on a K3 surface $X$. 
%Then:
%\begin{itemize}   
%\item $\Fix(\sigma^2)$  is either the disjoint union of two elliptic curves or the disjoint union of a smooth curve $C$ of genus $g\geq 0$ and 
%$j$ smooth rational curves;
%\item 
\end{theorem}
\begin{pro}\label{rel4}
Let $\sigma$ be a purely non-symplectic automorphism of order four on a K3 surface $X$. 
Then
$\Fix(\sigma)$ is the disjoint union of smooth curves and $n$ isolated points.
 Moreover, the following relations hold:
$$
n=2\alpha+4,\quad 4\alpha=r-l-2=2(10-l-m),
$$
where $\alpha=\sum_{C\subset \Fix(\sigma)} (1-g(C))$.
\end{pro}
\begin{proof} 
Since $\sigma$ is purely non-symplectic, then $\sigma^2$ is a non-symplectic involution.
Thus the fixed locus of $\sigma$ is the disjoint union of smooth curves and isolated points by Theorem \ref{inv}.
%By Theorem \ref{inv} the fixed locus of $\sigma^2$ is either empty, the disjoint union of two elliptic curves or the disjoint union of a curve of genus $g_2\geq 0$ and smooth rational curves. 
 The action of $\sigma$ at  a point in $\Fix(\sigma)$ can be locally diagonalized as follows (see \cite[\S 5]{Nikulin1}):
$$
A_{4,0}= \left(\begin{array}{cc}
i &0\\
0& 1
\end{array}
\right),\ \ 
A_{4,1}= \left(\begin{array}{cc}
-i &0\\
0& -1
\end{array}
\right).
$$
In the first case the point  belongs to a smooth fixed curve, while in the second case it is an isolated fixed point. 
 We will apply holomorphic and topological Lefschetz's formulas to obtain the last two relations in the statement.
The holomorphic Lefschetz number of $\sigma$ is
$$
L(\sigma)=\sum_{j=0}^{2}(-1)^j \tr(\sigma^*|H^j(X,\calO_X))=1-i,
$$ 
since $\sigma^*$ acts as multiplication by $i$ on $H^{2,0}(X)$.
By \cite[p. 567]{atiyahsinger} one obtains:
$$
L(\sigma)=\frac{n}{\det(I-\sigma^*|T_x)}+\frac{i-1}{2}\sum_{C\subset \Fix(\sigma)} (1-g(C))
=\frac{n}{\det(I-A_{4,1})}+\alpha \frac{i-1}{2},
$$
where $n$ is the number of isolated fixed points, $x$ is an isolated fixed point and $T_x$ denotes the tangent space at $x$.
Comparing the two formulas for $L(\sigma)$ 
we obtain the relation $n=2\alpha+4$. In particular this implies that the fixed locus of $\sigma$ (and thus that of $\sigma^2$) is not empty.
We now consider the topological Lefschetz fixed point formula
\begin{eqnarray*}
\chi(\Fix(\sigma))=\sum_{j=0}^{4}(-1)^j\tr(\sigma^*|H^j(X,\IR))
=2+\tr (\sigma^*|H^2(X,\IR)).
\end{eqnarray*}
Since $\tr (\sigma^*|H^2(X,\IR))=r-l$, then:
$$
\chi(\Fix(\sigma))=n+2\alpha=2+r-l
.$$
Using the relation $n=2\alpha+4$ we obtain the two expressions for $\alpha$ in the statement.  
\end{proof}
We now provide a similar result in case $\sigma^2$ is symplectic.
\begin{pro}\label{symp}
Let $\sigma$ be a non-symplectic automorphism of order four on a K3 surface $X$ such that $\sigma^2$ is symplectic. 
Then $\Fix(\sigma)$ contains $n\leq 8$ isolated fixed points and the possible values of the triple $(r,l,n)$ are 
\[
(6,8,0),\ (7,7,2),\ (8,6,4),\ (9,5,6),\ (10,4,8).
\] 
\end{pro}
\begin{proof}
The fixed locus of the symplectic involution $\sigma^2$ is given by $8$ points \cite{Nikulin1}.
Thus $\sigma$ fixes at most $n\leq 8$ points.
Moreover, the topological Lefschetz formula gives that $n=2+r-l$.
Since the invariant lattice of a symplectic involution has rank $14$ by \cite{Nikulin1}, then $l+r=14$. 
This gives the statement.  
\end{proof}

\begin{example}
Consider the following family of quartics in $\IP^3$:
$$a_1x_0^4+x_0^2(a_2x_1^2+a_3x_2x_3)+x_0x_1(a_4x_2^2+a_5x_3^2)+$$
$$x_1^2(a_6x_1^2+a_7x_2x_3)+x_2^2(a_{8}x_2^2+a_{9}x_3^2)+a_{10}x_3^4=0.
$$
The generic element $X_{a}$ of the family is a smooth quartic surface, hence a K3 surface, and carries the order four automorphism:
$$\sigma(x_0,x_1,x_2,x_3)=(x_0,-x_1,ix_2,-ix_3),$$
which has no fixed points and whose square fixes the eight intersection points between $X_a$ and the lines $x_0=x_1=0$, $x_2=x_3=0$.
Since the space of matrices in $\rm GL_4(\IC)$ commuting with $\sigma$ has dimension $4$, then the family has $10-4=6$ moduli.

On the other hand, the family of plane quartics of the form $f(x_0,x_1)+g(x_2,x_3)=0$, where $f,g$ are homogeneous of degree $4$, carries the order four automorphism $\tau(x_0,x_1,x_2,x_3)=(x_0,x_1,ix_2,ix_3)$, whose square is again symplectic and whose fixed locus contains exactly $8$ points.
The matrices commuting with $\tau$ are a space of dimension $8$ so we get $10-8=2$ moduli.

Finally consider the elliptic fibration
$$
y^2=x^3+a(t)x+b(t)
$$
with $a(t)=a_8t^8+a_4t^4+a_0, b(t)=b_{10} t^{10}+b_6 t^6+b_2 t^2$, $a_i,b_j\in\mathbb{C}$ and automorphism
$(x,y,t)\mapsto (-x,iy,it)$. The automorphism has symplectic square and 
leaves invariant the fibers over $t=0$ and $t=\infty$,
each of them contains two isolated fixed points which gives $n=4$.
Here the automorphisms of $\mathbb{P}^1$ commuting with the automorphism $t\mapsto it$ are a space of dimension 1, moreover we can apply to the equation the transformation
$(x,y)\mapsto(\lambda^2 x, \lambda^3 y)$ for a suitable $\lambda\in \mathbb{C}$ (and divide by $\lambda^6$). So we find $6-2=4$ moduli.
\end{example}

\begin{remark} The moduli space of K3 surfaces carrying a purely non-symplectic automorphism of order four with a  given action on the K3 lattice is known to be a complex ball quotient of dimension $m-1$, see \cite[\S 11]{DK}. The complex ball is given by:
$$B=\{[w]\in \IP(V): (w,\bar w)>0\},$$
where $V$ is the $i$-eigenspace of $\sigma^*$ in $T(\sigma^2)\otimes \IC$.
This implies that the Picard lattice of the K3 surface corresponding to the generic point of such space 
equals $S(\sigma^2)$ (see \cite[Theorem 11.2]{DK}). 
On the other hand, if the automorphism has symplectic square, then the period point of $X$ 
belongs to the $-1$-eigenspace in $H^2(X,\IZ)$, so that $\Pic(X)$ contains $S(\sigma)\oplus T(\sigma^2)$, $\rk\Pic(X)\geq 14$ and, given the action on the K3 lattice, the dimension of the moduli space is equal to $l-2$.  \end{remark}

The following results will be useful later.
 \begin{lemma}\label{lat}
Let $\sigma$ be a non-symplectic automorphism on a K3 surface $X$. Then the invariant lattice $S(\sigma)$ is a hyperbolic  sublattice of $\Pic(X)$.
\end{lemma}
\begin{proof} 
If $x\in S(\sigma)$, then $(x,\omega_X)=(\sigma^*(x),\sigma^*(\omega_X))=(x,\epsilon\omega_X)$, with $\epsilon\not=1$ since $\sigma$ is non-symplectic. Thus $x\in \Pic(X)=\omega_X^{\perp}\cap H^2(X,\IZ)$. 
%A similar argument shows that $S(\sigma^2)\subset \Pic(X)$ if $\sigma$ is purely non-symplectic.
By \cite[Theorem 3.1]{Nikulin1} the surface $X$ is algebraic, so that $\Pic(X)$ is hyperbolic by Hodge index theorem.
This implies that $S(\sigma)$ is a hyperbolic lattice since, given an effective class $x\in \Pic(X)$  with $x^2>0$, then $\sum_{i=1}^n(\sigma^*)^i(x)$ ($n$ is the order of $\sigma$) is a $\sigma^*$-invariant class with positive self-intersection.%Let $x\in T(\sigma^2)$ then $\sigma^2(x)+x=0$ on the other hand $T(\sigma^2)^{\vee}/T(\sigma^2)\cong S(\sigma^2)^{\vee}/S(\sigma^2)$ and so $0=\sigma^2(x)+x=2x$. Hence $T(\sigma^2)^{\vee}/T(\sigma^2)$ and also 
%$S(\sigma^2)^{\vee}/S(\sigma^2)$ is isomorphic to a sum of copies of $\IZ/2\IZ$ (see also \cite[Theorem 6.1]{kondoell}).  The relations in the statement are given in \cite[Theorem 4.2.2]{nikulinfactor} or \cite[Theorem 4.1]{ast}.
\end{proof}

\begin{lemma}\label{pic}
Let $\tau$ be a non-symplectic involution, $\pi:X\map Y:=X/\langle \tau\rangle$ 
be the quotient map and assume that the fixed locus of $\tau$ is the union of a 
curve $C$ and $j$ smooth rational curves $E_1,\dots,E_{j}$.
Then the invariant lattice $S(\tau)$ is generated by 
$\pi^*\Pic(Y)$ and by the classes of $E_1,\ldots, E_{j}$.
\end{lemma}
\begin{proof}
First observe that $\pi^*\Pic(Y)\otimes \IQ=S(\tau)\otimes \IQ$. In fact, if $y\in \Pic(Y)$
then $\pi^*(y)=z+\tau^*(z)$ is invariant by $\tau^*$. 
On the other hand, if $x\in S(\tau)$
then $x=(\pi^*\pi_*(x))/2\in\pi^*\Pic(Y)\otimes \IQ$. 
In particular  $\rho:=\rank S(\tau)=\rank \Pic(Y)$.
Moreover we have $\pi^*\Pic(Y)\cong \Pic(Y)(2)$ as a lattice.
 Since the quotient surface $Y$ is smooth and rational, then
$\Pic(Y)=H^2(Y,\IZ)$ is unimodular, so $[S(\tau):\pi^*\Pic(Y)]^2=2^\rho/2^d$.
Since  $2j=\rho-d$ by Theorem \ref{inv} we get $[S(\tau):\pi^*\Pic(Y)]=2^j$. 

Observe that the curves $\pi(E_i)$ are $(-4)$-curves in $Y$ and 
the sum of their classes is not equal to $-2K_Y$ (in fact 
$[\pi(C)+\sum_i \pi(E_i)]=-2K_Y$).
It follows by \cite[Lemma 2.2]{Sanchez} that any subset of  $\{E_1,\dots,E_j\}$
is not an even set.
This implies that the classes of the $E_i$'s are not contained in $\pi^*\Pic(Y)$ 
and give independent elements in $S(\tau)/\pi^*\Pic(Y)\cong \IZ/2\IZ^{\oplus j}$. 
Thus the index of the lattice generated by 
$\pi^*\Pic(Y)$ and the classes of $E_1,\ldots, E_j$ in the lattice $S(\tau)$ is one. 
\end{proof} 
\begin{lemma}\label{even1}
If $x\in  S(\sigma^2)$,  then  $x\cdot \sigma^*(x)$  is even.
\end{lemma}
\begin{proof}
We use Lemma \ref{pic} and write $x=\pi^*(y)+\sum a_i e_i$, where $a_i\in \IZ$ and $e_i$ is the class of $E_i$. Then $\sigma^*(x)=\pi^*\sigma^*(y)+\sum a_i\sigma^*(e_i)$,
where we use the fact that $\pi^*\sigma^*=\sigma^*\pi^*$ and we denote 
again with $\sigma$ the induced automorphism on $Y$. Then we get:
$$
x\cdot\sigma^*(x)=\pi^*(y)\cdot\pi^*\sigma^*(y)+\sum a_i e_i\cdot \sum a_i\sigma^*(e_i)+
 \pi^*(y)\cdot \sum a_i\sigma^*(e_i)+\pi^*\sigma^*(y)\cdot \sum a_i e_i
$$ 
and we have
$$
\pi^*(y)\cdot\pi^*\sigma^*(y)=2y\cdot\sigma^*(y),
$$
which is even, and clearly the product $\sum a_i e_i\cdot \sum a_i\sigma^*(e_i)$ is even too since the $E_i$'s are disjoint. Finally
$$
\pi^*(y)\cdot \sum a_i\sigma^*(e_i)=\pi^*(\sigma^*)^2(y)\cdot \sum a_i \sigma^*(e_i)=\pi^*\sigma^*(y)\cdot \sum a_i e_i,
$$
where the last equality follows from the fact that $\sigma^*$ is an isometry. This concludes the proof. \end{proof}

%If $\sigma(x)=x$, then the statement is obvious since $H^2(X,\IZ)$ is an even lattice.
%Otherwise, since  $x$ belongs to $S(\sigma^2)$, it is of the form $x=\frac{a+b}{n}$ for some positive integer $n$  where $a\in S(\sigma)$ and $b$ belongs to its orthogonal complement in $S(\sigma^2)$, where $\sigma^*=-\id$.
%Thus $x\cdot \sigma(x)=\frac{a^2-b^2}{n^2}=\frac{2a^2}{n^2}-x^2$, which is even.
We formulate the following useful lemma that generalizes \cite[Lemma 8.1]{jimmy}. 
Its proof in in fact the same as there, since it does not depend on the
order of the automorphism.
\begin{lemma}\label{jimmy}
Let $T=\sum_i R_i$ be a tree of smooth rational curves on a K3 surface $X$ such that each $R_i$ 
is invariant under the action of a purely non-symplectic automorphism $\sigma$ of order $k$. 
Then, the points of intersection of the rational curves $R_i$ are fixed by $\sigma$ and the action 
at one fixed point determines the action on the whole tree. 
\end{lemma}
\begin{proof}
Since $\sigma$ is purely non-symplectic, then it acts  on the vector space of holomorphic two forms of $X$ 
as the multiplication by a primitive $k$-th root of unity $\eta$. 
Let $p$ be the unique point of intersection of two rational curves $R_1$ and $R_2$. 
This is clearly fixed by $\sigma$, and $\sigma$ is locally given in $p$ by the matrix (see \cite[\S 5]{Nikulin1}):
$$\left(\begin{array}{cc}
\eta^{k+1-j} &0\\
0& \eta^{j}
\end{array}
\right),
$$
where the two eigenvectors correspond to the curves $R_1$ and $R_2$.
Assume that $R_2$ corresponds to the second eigenvector,
so that the action of $\sigma$ on the tangent space at $p$ in $R_2$ is $\mathbb{C}\rightarrow \mathbb{C},\,z\mapsto \eta^j z$. 
Let now $q$ be the point of intersection of $R_2$ with a rational curve $R_3\not=R_1$. 
Since $R_2\cong\mathbb{P}^1$ we can fix coordinates $(z_0:z_1)$ so that 
the action of $\sigma$ in $ \mathbb{P}^1\backslash \{q\}$ is $z_1/z_0\mapsto \eta^j z_1/z_0$ 
(in these coordinates $p=0$ and $q=\infty$). 
Thus the action on $\mathbb{P}^1\backslash \{p\}$ is given by $z_0/z_1\mapsto \eta^{k-j} z_0/z_1$. 
Since $\sigma$ multiplies holomorphic two forms by $\eta$, 
then its local action at $q$ is given by the matrix: 
$$\left(\begin{array}{cc}
\eta^{k-j} &0\\
0& \eta^{j+1}
\end{array}
\right),
$$
so it is uniquely determined.
\end{proof}
\begin{remark}
In the case of an automorphism of order $4$, with the assumption of Lemma \ref{jimmy}, 
the local actions at the intersection points of the curves $R_i$ appear in the following order:
$$\ldots, \left(\begin{array}{cc}
1 &0\\
0& i
\end{array}
\right), \left(\begin{array}{cc}
-i &0\\
0& -1
\end{array}
\right), \left(\begin{array}{cc}
-1 &0\\
0& -i
\end{array}
\right), \left(\begin{array}{cc}
i &0\\
0& 1
\end{array}
\right), \left(\begin{array}{cc}
1 &0\\
0& i
\end{array}
\right), \ldots
$$
%\diag(1,i), \diag(-i,-1), \diag(-1,-i), \diag(i,1), \diag(1,i),......
%$$ 
\end{remark}

 %%%%%%%%%%%%%%%%%%%%%%%%%%%%%%%%%%%%%%%%%%%%%%% 
\section{Elliptic fibrations}
In this section we will study elliptic fibrations on K3 surfaces carrying a purely non-symplectic automorphism of order four. 
The following result is proved by means of an argument contained in the proof of \cite[Proposition 2.9]{DKe}.
  \begin{pro}\label{ellinv}
  Let $\sigma$ be an automorphism of a K3 surface $X$.
 If the rank of the invariant lattice of $\sigma^*$ in $H^2(X,\IZ)$ is bigger than $4$, then there exists a $\sigma$-invariant elliptic fibration $\pi:X\map \IP^1$.
 \end{pro}
 \begin{proof} As before, we denote by $S(\sigma)$ the invariant lattice of $\sigma^*$ in $H^2(X,\IZ)$.
 By \cite[Corollary 2, pag. 43]{Se}, since $\rk S(\sigma)\geq 5$,  there exists a primitive isotropic vector $x\in S(\sigma)$.
 After applying a finite number of reflections with respect to $(-2)$-curves we obtain a nef class $x'$ which is uniquely determined by $x$ (see \cite[\S 6, Theorem 1]{PSS}).  Observe that $x'$ is primitive and $x'^2=0$.
It is easy to see that $\sigma^*$ acts 
on the orbit of $x$ with respect to the reflection group. Since $x'$ is the unique 
nef member in the orbit and any automorphism preserves nefness, then
$\sigma^*(x') = x'$. The morphism associated to $x'$ is a  $\sigma$-invariant elliptic fibration on $X$.
%([SD] or [Re, Th. 3.8]).  
\end{proof}

 Let $\pi:X\map \IP^1$ be an elliptic fibration such that any of its fibers is invariant for $\sigma$ and contains at least a fixed point. The last assumption is not necessary if the fibration is jacobian: indeed if $\sigma$ is fixed points free on the generic fiber, then it acts as a translation on it and it can be easily proved, by writing explicitly a holomorphic 2-form, that the automorphism would be symplectic. 
 If $\sigma$ has order four, then the generic fiber of $\pi$ contains two fixed points for $\sigma$ and four fixed points of $\sigma^2$.
Thus $\pi$ has two bisections (not necessarily irreducible): a curve $E_{\sigma}\subset \Fix(\sigma)$ and a curve $E_{\sigma^2}\subset \Fix(\sigma^2)$.
 We now describe the singular fibers of the elliptic fibration and the action of $\sigma$ on them.

\begin{pro}\label{ell} Let $X$ be a K3 surface with a non-symplectic order four automorphism $\sigma$ and $\pi:X\map \IP^1$ be an elliptic fibration such that $\sigma$ preserves each fiber of $\pi$ and has a fixed point on it. 
Then the singular fibers of $\pi$ are of the following Kodaira types:
\begin{enumerate}[$\bullet$]
\item ${\rm III:}$ $R_1\cup R_2$, where either\\
a) the $R_i$'s are exchanged by $\sigma$, $E_{\sigma^2}$ intersects each $R_i$ at one point and $E_{\sigma}$ intersects in $R_1\cap R_2$ or \\
%($n=k=0$) or\\
b) the $R_i$'s are $\sigma$-invariant, $E_{\sigma}$ intersects each $R_i$ at one point and $E_{\sigma^2}$ intersects in $R_1\cap R_2$. 
%($n=1, k=0$).
\item ${\rm I_0^*:}$ $2R_1+R_2+R_3+R_4+R_5$, where either\\
a)  $R_2, R_3$ are $\sigma$-invariant (intersected by $E_{\sigma}$) and $R_4, R_5$ are exchanged by $\sigma$ (intersected by $E_{\sigma^2}$) or\\
% ($n=2, k=0$) 
b)  $R_2,\dots, R_5$ are permuted by $\sigma$, $E_{\sigma}$ and $E_{\sigma^2}$ intersect $R_1$. 
%($n=1, k=0$).
\item ${\rm III^*:}$ $R_1+2R_2+3R_3+4R_4+2R_5+3R_6+2R_7+R_8$, where either\\
%a) $\sigma$ preserves each irreducible component of the fiber, $R_2, R_4, R_7\subset \Fix(\sigma)$, $E_{\sigma}$ intersects $R_1, R_8$ and  $E_{\sigma^2}$ intersects $R_5$  or\\
%($n=1, k=3$)
a) $\sigma$ preserves each irreducible component of the fiber, $R_4\subset \Fix(\sigma)$, $R_2, R_7$ contain two isolated fixed points, $E_{\sigma}$ intersects $R_1, R_8$, $E_{\sigma^2}$ intersects $R_5$ or\\
%($n=5, k=1$) 
b) $\sigma$ exchanges the two branches of the fiber, $E_{\sigma^2}$ intersects $R_1, R_8$ and $E_{\sigma}$ intersects $R_5$.\\
%($n=2, k=0$)
\end{enumerate}
\end{pro}
\vspace{-0.6cm}

\begin{proof}
By the previous argument,  the restriction of $\sigma$ to the generic fiber of $\pi$ has order four and two fixed points. Thus any smooth fiber of $\pi$ has $j$-invariant equal to $1$. By the Kodaira classification it follows that the singular fibers of $\pi$ are either of type ${\rm I_0^*, III}$,  or ${\rm III^*}$. We now analyze the possible actions of $\sigma$ on these fibers.
 
If $F$ is a reducible fiber of type ${\rm I_0^*}$, then the component $R_1$ is clearly $\sigma$-invariant.
Observe that $R_1$ is not fixed by $\sigma$, since otherwise each $R_i$, $i=2,\dots,5$, should contain a fixed point for $\sigma$  in the intersection with either $E_{\sigma^2}$ or $E_{\sigma}$. This is absurd because $\sigma$ exchanges the two (distinct) points in $F\cap E_{\sigma^2}$. 
Thus $\sigma$ has either order two or four on $R_1$.
If $\sigma^2=\id$ on $R_1$, then each $R_i$, $i=2,\dots,5$ contains a fixed point of $\sigma^2$ in the intersection with either $E_{\sigma^2}$ or $E_{\sigma}$, thus we are in case ${\rm I_0^*} a)$.
If $\sigma$ has order $4$ on $R_1$, then $\sigma$ permutes the curves $R_i$, $i=2,\dots,5$  and $R_1$ contains two fixed points for $\sigma$ in the intersection with $E_{\sigma}$ and $E_{\sigma^2}$, giving case ${\rm I_0^*} b)$.

If $F$ is a fiber of type ${\rm III^*}$, then $R_4$ and $R_5$ are clearly $\sigma$-invariant.
If $\sigma$ preserves each irreducible component of $F$, then $R_4\subset \Fix(\sigma)$ (since it contains $3$ fixed points) and, since $E_{\sigma^2}$ contains at most a fixed point, $E_{\sigma}$ intersects $R_1, R_8$ and $E_{\sigma^2}$ intersects $R_5$. 
By Lemma \ref{jimmy}, the only curve fixed by $\sigma$ is $R_4$, thus the curves $R_2$ and $R_7$ contain each two isolated fixed points. 
This gives case ${\rm III^*} a)$.
Otherwise, if $\sigma$ exchanges the two branches of the fiber, then $\sigma^2=\id$ on $R_4$, $E_{\sigma^2}$ intersects $R_1$ and $R_8$ in two points exchanged by $\sigma$ and $E_{\sigma}$ intersects $R_5$.
The case of a fiber of type ${\rm III}$ can be discussed in a similar way. 
\end{proof}
We will denote by $g_{\sigma}$ and $g_{\sigma^2}$ the genus of $E_{\sigma}$ and $E_{\sigma^2}$ respectively, by $n$ the number of isolated points in $\Fix(\sigma)$, by $k$ the number of smooth rational curves in $\Fix(\sigma)$ and by $2a$ the number of smooth rational curves in $\Fix(\sigma^2)$ exchanged by $\sigma$.
\begin{cor}\label{cor} Under the hypotheses of Proposition \ref{ell}, we have the following possibilities for the invariants defined above.
 \begin{itemize}
\item If $E_{\sigma}$ is irreducible and $E_{\sigma^2}$ is reducible (hence the union of two smooth rational curves exchanged by $\sigma$):
$$
\begin{array}{c|ccc|l}
g_{\sigma} &  n & k & a & \mbox{reducible fibers} \\
\hline
3 & 0 & 0 & 1 & 8\, {\rm III }\, a)\\
\hline
 2  & 2 & 0 & 1 & 6\,{\rm III}\, a)+{\rm I_0^*}\, a)\\
 & 2 & 0 & 2 & 5\,{\rm III} \,a)+{\rm III^*} \,b)\\
\hline
1  & 4 & 0 & 1   & 4\,{\rm III}\, a)+2\,{\rm I_0^*}\, a)\\
& 4 & 0 & 2   & 3\,{\rm III}\,a)+{\rm I_0^*} \,a)+{\rm III^*} \,b)\\
& 4 & 0 & 3 & 2\,{\rm III}\,a)+2\,{\rm III^*}\,b)\\
\hline
 0& 6& 1 & 1   & 2\,{\rm III}\,a)+3\,{\rm I_0^*}\,a)\\
  & 6& 1 & 2   & {\rm III}\,a)+2\,{\rm I_0^*}\,a)+{\rm III^*}\,b)\\
 & 6& 1 & 3   & {\rm I_0^*}\,a)+2\,{\rm III^*}\,b)\\
\end{array}
$$
\end{itemize}

\begin{itemize}
\item If $E_{\sigma}$ is reducible and $E_{\sigma^2}$ is irreducible: $a=0$ and  $$
\begin{array}{c|cc|l}
g_{\sigma^2} &  n & k & \mbox{reducible fibers} \\
\hline
3 & 8 & 2  & 8\,{\rm III}\,b)\\
\hline
2   & 8 & 2 & 6\,{\rm III}\, b)+{\rm I_0^*}\,a)\\
 & 10&3  & 5\,{\rm III}\, b)+{\rm III^*}\,a)\\
\hline
 1  & 8 & 2 & 4\,{\rm III}\,b)+2\,{\rm I_0^*}\,a)  \\
    & 10 & 3 & 3\,{\rm III}\,b)+{\rm I_0^*}\,a)+{\rm III^*}\,a)\\
  &12 & 4 &  2\,{\rm III}\,b)+2\,{\rm III^*}\,a)\\
\hline
0 & 8 & 2 & 2\,{\rm III}\,b)+3\,{\rm I_0^*}\,a)\\
 & 10 & 3 & {\rm III}\,b)+2\,{\rm I_0^*}\,a)+{\rm III^*}\,a)\\
 & 12 & 4 & {\rm I_0^*}\,a)+2\,{\rm III^*}\,a)
\end{array}
$$
\end{itemize}

\begin{itemize}
\item If $E_{\sigma}$ and $E_{\sigma^2}$ are both irreducible:
$$
\begin{array}{cc|ccc|l}
g_{\sigma} & g_{\sigma^2} & n & k & a & \mbox{reducible fibers} \\
\hline
2 & 0 & 2 & 0 & 0 & 6\,{\rm III}\,a)+2\,{\rm III}\,b)\\
\hline
1 & 1 & 4 & 0 & 0 & 4\,{\rm III}\,a)+4\,{\rm III}\,b) \\
  &     & 4 & 0 & 0 & 4\,{\rm I_0^*}\,b)\\
\hline
1  & 0 &  6 & 1 &0 & 3\,{\rm III}\,a)+{\rm I_0^*}\,b)+{\rm III^*}\,a)\\
    &    & 6 & 1 & 0 & 4\,{\rm III}\,a)+{\rm III}\,b)+{\rm III^*}\,a)\\
  &    & 4 & 0 & 0 & 4\,{\rm III}\,a)+2\,{\rm III}\,b)+{\rm I_0^*}\,a)\\
     &    & 4 & 0 & 1 & 3\,{\rm III}\,a)+2\,{\rm III}\,b)+{\rm III^*}\,b)\\
\hline 
0& 2 & 6 & 1 & 0 & 2\,{\rm III}\,a)+6\,{\rm III}\,b)\\
  &    & 6 & 1 & 0 & 4\,{\rm III}\,b)+2\,{\rm I_0^*}\,b)\\
\hline  
0 & 1     & 8 & 2 & 0 & 2\,{\rm III}\,a)+3\,{\rm III}\,b)+{\rm III^*}\,a)\\
 &  & 6 & 1 & 0 & 2\,{\rm III}\,a)+4\,{\rm III}\,b)+{\rm I_0^*}\,a)\\
 &     & 6 & 1 & 1 & {\rm III}\,a)+4\,{\rm III}\,b)+{\rm III^*}\,b) \\
\hline
0 & 0   & 10 & 3 & 0 & 2\,{\rm III}\,a)+2\,{\rm III^*}\,a)\\
 &   & 8 & 2 & 0 & 2\,{\rm III}\,a)+{\rm III}\,b)+{\rm I_0^*}\,a)+{\rm III^*}\,a)\\
  &   & 8 & 2 & 1 & {\rm III}\,a)+{\rm III}\,b)+{\rm III^*}\,a)+{\rm III^*}\,b)\\
   && 6 & 1 & 0 & 2\,{\rm III}\,a)+2\,{\rm III}\,b)+2\,{\rm I_0^*}\,a)\\
  &   & 6 & 1 & 1 & {\rm III}\,a)+2\,{\rm III}\,b)+{\rm I_0^*}\,a)+{\rm III^*}\,b)\\
 &   &  6 & 1 & 2 & 2\,{\rm III}\,b)+2\,{\rm III^*}\,b)
  \end{array}$$
 % \end{itemize}

%\begin{itemize}

\item If both $E_{\sigma}$ and $E_{\sigma^2}$ are reducible: $n=8, k=2, a=1$ and the reducible fibers  are $4, {\rm I_0^*}$ of type $a)$.
\end{itemize}
\end{cor}
\begin{proof}
Observe that the restrictions of $\pi$ to $E_{\sigma}$ and to $E_{\sigma^2}$ are double covers of $\IP^1$.
If $E_{\sigma}$ (or $E_{\sigma^2}$) is irreducible, then it contains $2g_{\sigma}+2$ (or $2g_{\sigma^2}+2$) ramification points.
Since a smooth fiber of $\pi$ contains exactly $4$ fixed points for $\sigma^2$, then such ramification points belong to singular fibers of $\pi$, which are classified in Proposition \ref{ell}.
The ramification points of $E_{\sigma}$  belong either to a fiber of type ${\rm III}\, a)$, ${\rm I_0^*}\,b)$ or ${\rm III^*}\,b)$, 
while the ramification points of $E_{\sigma^2}$ belong either to a fiber of type  ${\rm III}\, b)$, ${\rm I_0^*}\, b)$, or ${\rm III^*}\,a)$.
This implies that $g_{\sigma}\leq 3$ (or $g_{\sigma^2}\leq 3$) since otherwise the Euler-Poincar\'e characteristic of the singular fibers 
would give at least $e({\rm III})(2g_{\sigma}+2)=3(2g_{\sigma}+2)>24=e(X)$ (similarly for $g_{\sigma^2}$). 
If $E_{\sigma}$ and $E_{\sigma^2}$ are both irreducible, this implies that $g_{\sigma}, g_{\sigma^2}\leq 2$. We obtain the tables in the statement by enumerating all cases which are compatible with Proposition \ref{ell} and Proposition \ref{rel4}.
\end{proof}
The following result allows, in some cases, to prove that a given elliptic fibration is $\sigma$-invariant.
\begin{lemma}\label{ineq}
Let $X$ be a K3 surface with an automorphism $\sigma$ and $\pi:X\map \IP^1$ be an elliptic fibration whose fibers have class $f$.
If $\sigma^*$ fixes a class  $x\in \Pic(X)$ with $x^2>0$, then
$$(f\cdot\sigma^*(f))x^2\leq 2(x\cdot f)^2.$$

Moreover, if in addition $\pi$ is jacobian and there is a section $s$ of $\pi$ such that $s\cdot x=0$, then the  following holds
$$x^2\leq \frac{2(x\cdot f)^2}{f\cdot \sigma^*(f)+1}.$$
\end{lemma}
\begin{proof}
Let $M$ be the sublattice of $\Pic(X)$ generated by $x$ and $f+\sigma^*(f)$. By Hodge index theorem we have that $\det(M)=2(x^2(f\cdot \sigma^*(f))-2(x\cdot f)^2)\leq 0$ . This gives the first inequality.
The second inequality follows from a similar argument with the lattice generated by $x, f+\sigma^*(f)$ and the class of the section not intersecting $x$.
\end{proof}

\begin{theorem}\label{jac}
Let $\sigma$ be a purely non-symplectic automorphism of order $4$ on a K3 surface $X$ such that  $\Pic(X)=S(\sigma^2)\cong U\oplus L$ and $\sigma^2$ fixes a curve $C$ of genus $g>1$. Then  $X$ carries a jacobian elliptic fibration $\pi:X\map \IP^1$ whose fibers are $\sigma^2$-invariant.  Moreover:
\begin{itemize}
\item If $L$ is isomorphic to a direct sum of root lattices of types $A_1, D_{4n}, E_7$ or $E_8$, then $\pi$ has reducible fibers described by $L$ and a unique section $E\subset \Fix(\sigma^2)$. Moreover, $\pi$ is $\sigma$-invariant if $g> 4$ and the genus of a curve in $\Fix(\sigma)$  is at most $2$.
 
%\noindent Moreover, if $\Fix(\sigma^2)$ contains a curve $C$ of genus $g>1$, then:
%\begin{enumerate}[a)]
%\item  $\sigma^2$ preserves each fiber of $\pi$, $C$ intersects the generic fiber at three points and $E\subset \Fix(\sigma^2)$;
%\item  $\pi$ is $\sigma$-invariant if $g> 4$;
%\item the genus of a curve in $\Fix(\sigma)$   is $\leq 2$.
%\end{enumerate}
\item If $L$ is not isomorphic to a direct sum of root lattices, then $\pi$ has two sections $E,E'\subset \Fix(\sigma^2)$ and $C$ intersects each fiber in two points. Moreover, $\pi$ is $\sigma$-invariant if $g> 2$ and the genus of a curve in $\Fix(\sigma)$  is at most $2$ if $a=0$.

\end{itemize}
In both cases the automorphism $\sigma^2$ acts as an involution on the simple components of the reducible fibers of $\pi$  and on the fibers of types ${\rm I^*}_{\!\!4n}, {\rm III^*, II^*}$  as in Figure \ref{action}, where $\sigma^2$ acts identically on dotted components and as an involution on the other ones. 
\end{theorem}
\begin{proof} 
The first half of the statement follows from \cite[Lemma 2.1]{kondoell}.  
If $\sigma^2$ fixes a curve $C$ of genus $g>1$, then this curve is transversal to the fibers of $\pi$. This implies that $\sigma^2$ preserves each fiber of $\pi$ and has $4$ fixed points on it. 
If $L$ is a root lattice, then $\pi$ has a unique section $E$ since its Mordell-Weil group $MW(\pi)\cong\Pic(X)/U\oplus L$ is trivial \cite[Theorem 6.3]{SS}.
This implies that $E$ is fixed by $\sigma^2$ and that $C$ intersects each fiber in three points.

Let $x$ be the class of $C$ and $f$ be the class of a fiber of $\pi$.
If $f\not=\sigma^*(f)$, then clearly $f\cdot \sigma^*(f)\geq 2$. 
It follows from Lemma \ref{ineq} that 
$2g-2=x^2\leq  \frac{2(x\cdot f)^2}{f\cdot \sigma^*(f)+1}\leq 6$, thus $g\leq 4$.
Observe that, if $\sigma$ fixes a curve $C$ of genus $g>1$, then $f\not=\sigma^*(f)$ since otherwise the generic 
fiber would contain at most $2$ fixed points by $\sigma$. Moreover in this case $f\cdot \sigma^*(f)\geq 4$ since each fiber contains at least $3$ fixed points and the intersection $f\cdot \sigma^*(f)$ is even by Lemma \ref{even1}. This implies $g\leq 2$ by Lemma \ref{ineq}.

If $L$ is not a root lattice, then the Mordell-Weil group of $\pi$ is not trivial, so that $\pi$ has at least two sections $E,E'\subset \Fix(\sigma^2)$.
Thus $C$ intersects each fiber in two points. In this case, if $f\not=\sigma^*(f)$, then $g\leq 2$ by  Lemma \ref{ineq}. 
If $\sigma$ fixes $C$, then either $\pi$ is $\sigma$-invariant or $f\cdot \sigma^*(f)\geq 2$.
In the first case, since $\sigma$ has two fixed points in each fiber, $\sigma(E)=E'$, so that $a\geq 1$.
In the second case Lemma \ref{ineq} implies, as before, that $g\leq 2$. 
\end{proof}
 
\begin{figure}[ht!]
\begin{center}
\includegraphics[scale=0.29]{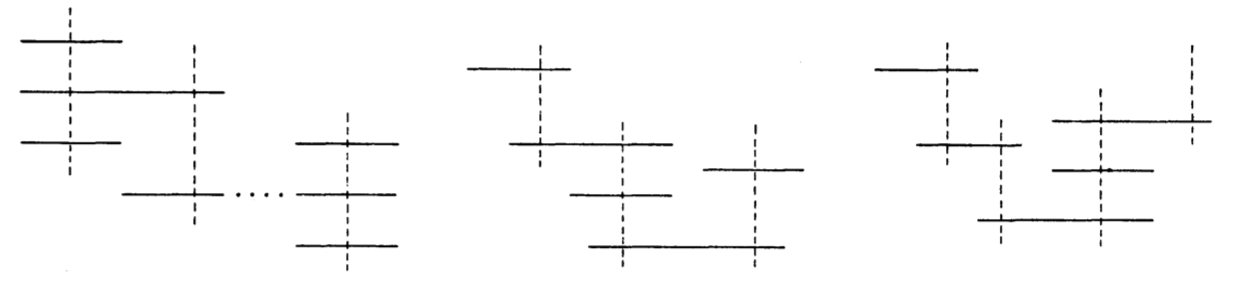}
\end{center}
\caption{Action of $\sigma^2$ on reducible fibers of types ${\rm I^*}_{\!\!4n}, {\rm III^*, II^*}$}
\label{action}
\end{figure}
\begin{remark}\label{split}
By \cite[Corollary 1.13.5]{nikulinintegral} an even indefinite lattice $S$ is isomorphic to a direct orthogonal sum $U\oplus L$, if $\rk(S)\geq 3+\ell(S)$, where $\ell(S)$ is the minimal numbers of generators of the discriminant group of $S$. In case $S=S(\sigma^2)$ this condition is equivalent to ask that the number $j$ of smooth rational curves fixed by $\sigma^2$ is at least $2$ by Theorem \ref{inv}.
\end{remark}

\section{$\Fix(\sigma)$ contains an elliptic curve}
We now assume that $\sigma$ fixes an elliptic curve $C$. In this case the K3 surface $X$ has an elliptic fibration $\pi_C:X\map \IP^1$ having $C$ as a smooth fiber. Observe that  all curves fixed by $\sigma^2$, since they are disjoint from $C$, are contained in the fibers of $\pi_C$. In particular $\alpha\geq 0$.
We will now classify the reducible fibers of $\pi_C$. 
%For the following result, see also \cite[\S 4.2]{nikulinfactor}.

\begin{theorem}\label{g1}
Let $\sigma$ be a purely non-symplectic order four automorphism on a K3 surface $X$ with 
$\Pic(X)=S(\sigma^2)$ and $\pi_C:X\map \IP^1$ be an elliptic fibration with a smooth fiber $C\subset \Fix(\sigma)$.
Then $\sigma$ preserves $\pi_C$ and acts on its base as an order four automorphism with two fixed points corresponding to the  fiber $C$ and a fiber $C'$  which is either smooth, of Kodaira type ${\rm I}_{4M}$ or ${\rm IV^*}$. The corresponding invariants of $\sigma$ are given in Table \ref{g=1}. 

%Examples for all
%the cases in the table, except for the cases $(r,k,a)=(10,1,0)$ with fiber ${\rm I_8}$ and  $(r,k,a)=(8,0,1)$ with fiber ${\rm IV^*}$, are given in Examples \ref{wei}, \ref{quartic} and \ref{hyp}.
 
%%%%%%%%%%%%%%%%%%%%%%%%%%%%%%%%%%%%%%%
\begin{table}[ht] 
$$ 
\begin{array}{ccc|ccc|c}
m & r & l &  n &k & a & \mbox{type of } C'  \\
\hline
6& 6 & 4 & 4 & 0 & 0  & {\rm I_0}\\
\hline
5 & 7 & 5 & 4 & 0 & 0  & {\rm I_4} \\
%&   & \mbox{quartic 2 nodes }   \\
\hline
 4 & 10 & 4 &  6 & 1 & 0 & {\rm IV^*}   \\
 %& &      \mbox{quartic 3 nodes } (k=0),\ \mbox{cubic+line } (k=1)  \\
 4 & 8 & 6 &   4&0 &   1 & {\rm I}_8 \ or\ {\rm IV^*}     \\
  %& &  \mbox{hyp. 2 nodes},\  \mbox{quar. 1 node+1 cusp}\\
  \hline
 3 & 9 & 7 &   4& 0 & 2 & {\rm I_{12}}  \\
   %3 &   13  & 3   &   8 &   2 &   0  &  \\ 
   %& &  \mbox{quar. 2 cusps },\  \mbox{hyp. 1 node+1 cusp}, I_{12}+12I_1 \\
   \hline  
   2   & 10 & 8 &  4 & 0 &  3 & {\rm I_{16}} \\
        %& &  \mbox{hyp. 2 cusps}, I_{16}+8I_1\\
         \end{array}$$
\vspace{0.2cm}
\caption{The case $g=1$}\label{g=1}
\end{table}
\end{theorem}
\begin{proof} We first observe that $\sigma^2$ is not the identity on the base of $\pi_C$, since otherwise  
it would act as the identity on the tangent space at a point of $C$, contradicting the fact that $\sigma^2$ is non-symplectic. Hence $\sigma$ has order four on $\IP^1$ and has two fixed points, corresponding to $C$ and another fiber $C'$.
If  $C'$  is irreducible, then $\alpha=0$ and $n=4$ by Proposition \ref{rel4}, which implies that $C'$ is smooth elliptic and  $\sigma$ has order two on it.
 
We now assume that $C'$ is reducible and we classify the possible Kodaira types for it (see also \cite[\S 4.2]{nikulinfactor}).
Since $n\geq 4$ by Proposition \ref{rel4}, then $C'$ contains at least two (disjoint) smooth rational curves fixed by $\sigma^2$. This excludes the Kodaira types ${\rm  I_2, I_3, III, IV}$ for $C'$.
Since $S(\sigma^2)=\Pic(X)$, then any smooth rational curve is invariant for $\sigma^2$.
Moreover, observe that if a component of $C'$  is ``external'', i.e. it only intersects one other component, then it is fixed by $\sigma^2$. 
%since otherwise it should contain a fixed point outside of any curve fixed by $\sigma^2$. 

If $C'$ is of type ${\rm I}_N^*$, then the four external components of $C'$ are fixed by $\sigma^2$ by the previous remark and the same holds for the multiplicity two components intersecting them, since they contain at least $3$ fixed points. This gives a contradiction since the fixed curves of $\sigma^2$ do not intersect. The case when $C'$ is either of type ${\rm II^*}$ or ${\rm III^*}$ can be excluded by a similar argument.

%either $\sigma^2$ preserves each component of $C'$ or $\sigma^2$ acts on the components of $C'$ as an involution.
%In the first case the four external components of $C'$ are fixed by $\sigma^2$ by the previous remark and the same holds for the multiplicity two components intersecting them, since they contain at least $3$ fixed points, giving a contradiction.
%In the second case $\sigma^2$ exchanges two multiplicity one components $C'_1, C'_2$ of $C'$. If $C'_1$ and $C'_2$ intersect the same  multiplicity two component $D$, then $\sigma^2(D)=D$ but $D$ contains at most a fixed point of $\sigma^2$, a contradiction.
%Finally, if $C'_1, C'_2$ intersect distinct multiplicity two components, then $\sigma^2$ acts as a simmetry with respect to the central component of $C'$, so that  $n=0$, giving a contradiction again.

%If $C'$ is either of type ${\rm II^*}$ or ${\rm III^*}$, then as before one observes that both the central component (of multiplicity $6$ and $4$ respectively) and the external component intersecting it (of multiplicity $3$ and $2$ respectively)  are fixed by $\sigma^2$, giving a contradiction.

%then $\sigma^2$ preserves each component, since the set of components  has no order four automorphism. As before, this implies that both the central component (of multiplicity $6$ and $4$ respectively) and the external component (of multiplicity $3$ and $2$ respectively) intersecting it  are fixed by $\sigma^2$, giving a contradiction.

If $C'$ is of type ${\rm IV^*}$, then the central component of multiplicity $3$ is clearly invariant for $\sigma$.
If  the central component is fixed by $\sigma$, then $k=1$, $a=0$ and $n=6$  by Proposition \ref{rel4}.
Otherwise two branches of the fiber are exchanged, and the same Proposition gives $k=0$, $n=4$, $a=1$. 
 
 Finally we assume that $C'$ is of type ${\rm I}_N$, $N\geq 4$. By the previous remark $C'$ contains at least two components fixed by $\sigma^2$, this implies that all components are preserved by $\sigma^2$ and a component which is not fixed intersects two fixed ones. Moreover, it follows from Proposition \ref{rel4} that the number of components of $C'$ in $\Fix(\sigma^2)$ is even, since it equals 
$j=k+ n/2+2a=2+2\alpha+2a.$ Thus  $C'$ is of type ${\rm I}_{4M}$ for some positive integer $M$.

If $a=0$, then $n=4M-2k$. Since $n=2k+4$ by Proposition \ref{rel4}, this gives $k=M-1$ and $n=2M+2$. On the other hand by applying Lemma \ref{jimmy} we obtain for $M\geq 2$ that $k=4M/4=M$, which is clearly a contradiction. So if $a=0$ then $M=1$. 

If $a=M-1$, $M\geq 2$, then $\sigma$ acts as an order two symmetry on the set of components of $C'$, so that $k=0$ and $n=4$.
Observe that $\sigma$ can not act as a rotation of ${\rm I}_{4M}$ because otherwise $n=0$, contradicting the fact that $n\geq 4$.

If $M=1$ we have a fiber ${\rm I}_4$. Since $n\geq 4$ and the fixed points are contained in the fiber $C'$
the only possibility is that $k=0$ and $n=4$. Observe that the case with $a=1$ is not possible,
otherwise the fiber $C'$ would contain two isolated fixed points for $\sigma^2$ which is not possible, so $a=0$. 
As before one remarks that $\sigma$ can not act as a rotation of ${\rm I}_4$
since otherwise $n=0$ contradicting the fact that $n\geq 4$. 
\end{proof}

\begin{example}\label{wei}
We now assume that the fibration $\pi_C:X\map \IP^1$ in Theorem \ref{g1} has a $\sigma$-invariant section. Then a Weierstrass equation for the fibration is of  the following form:
$$
y^2=x^3+\alpha(t)x+\beta(t),
$$  
where $\alpha(t)=ft^8+at^4+b$, $\beta(t)=gt^{12}+ct^8+dt^4+e$ and  
$$\sigma(x,y,t)=(x,y,it).$$ 
The fibers preserved by $\sigma$ are over $0, \infty$ and the action at infinity is 
$$(x/t^4,y/t^6,1/t)\mapsto (x/t^4,-y/t^6,-i/t).$$ 
The discriminant polynomial of $\pi_C$ is: 
$$\Delta(t):=4\alpha(t)^3+27 \beta(t)^2=g_1t^{24}+g_2t^{20}+g_3t^{16}+g_4t^{12}+O(t^8),$$
where 
$$
g_1=4f^3+27g^2,\quad g_2=12f^2a+54gc,\quad g_3=12f^2b+54gd+12fa^2+27c^2,
$$
$$
g_4=24fab+4a^3+54ge+54cd
.$$
For a generic choice of the coefficients of $\alpha(t)$ and $\beta(t)$ the fibration has 24 fibers of type ${\rm I}_1$ over the zeros of $\Delta(t)$, $\sigma$ fixes pointwisely the fiber over $0$ and it acts as an involution on the fiber over $\infty$ (both fibers are smooth). 
If $g_1=0$ the fibration acquires a fiber of type ${\rm I}_4$ at $\infty$ by a generic choice of the parameters, 
if $g_1=g_2=g_3=0$ we generically get a fiber of type ${\rm I}_{12}$ and for $g_1=g_2=g_3=g_4=0$ we get a fiber of type ${\rm I}_{16}$. 
 If $g_1=g_2=0$ one gets two possible solutions: if $f=g=0$ the fibration acquires a fiber of type 
${\rm IV^*}$, otherwise it gets a fiber of type ${\rm I}_8$. 
\end{example}
More examples for the case $g=1$ will be given in Examples \ref{quartic} and \ref{hyp}.
 
\section{$\Fix(\sigma)$ contains a curve of genus $>1$}
We now assume that $\Fix(\sigma)$ contains a curve $C$ of genus $g>1$. 
By Proposition \ref{rel4} we have
$$\Fix(\sigma^2)=C\cup (E_1\cup\cdots\cup E_k)\cup (F_1\cup F'_1 \cup\cdots \cup F_a\cup F'_a)\cup (G_1\cup \cdots\cup G_{n/2}),$$
$$\Fix(\sigma)=C\cup E_1\cup\cdots\cup E_k\cup\{p_1,\dots,p_n\},$$
where $E_i, F_i, G_i$ are smooth rational curves such that
$\sigma(F_i)=F'_i,\ \sigma(G_i)=G_i$ and each $G_i$ contains exactly two isolated fixed points of $\sigma$.

\begin{lemma}\label{rel4a}
$k\leq r+m-8$,\ $l-m=2a$.
\end{lemma}
\begin{proof}
The curves $C,\ E_i,\ F_i\cup F'_i,\ G_i$ are $\sigma$-invariant and are orthogonal to each other, thus  their classes in $\Pic(X)$ give independent elements in $S(\sigma)$.
Then $r=\rank (S(\sigma))\geq 1+k+a+n/2$ and this gives the inequality.
An easy computation shows that 
$$\mathcal X(\Fix(\sigma^2))-\mathcal X(\Fix(\sigma))=4a.$$
On the other hand, by Proposition \ref{rel4} and topological Lefschetz fixed point formula applied to $\sigma^2$:
$$\mathcal X(\Fix(\sigma))=24-2m-2l,\ \mathcal X(\Fix(\sigma^2))=24-4m.$$
Comparing these equalities, we obtain the statement.
\end{proof}

\begin{theorem}\label{propocurve}
Let $X$ be a K3 surface and $\sigma$ be a purely non-symplectic automorphism of order four on it such that $\Pic(X)=S(\sigma^2)$.
If $\Fix(\sigma)$ contains a curve of genus $g> 1$ then the invariants associated to $\sigma$ are as in Table \ref{g>1}. All cases in the Table do exist, see Example \ref{quartic} and Example \ref{hyp}.

\begin{table}[ht] 
$$ 
\begin{array}{ccc|ccccccl}
m & r & l &  n &k & a & g \\
\hline
7 & 1 & 7 & 0  &0 & 0 & 3 \\
%\mbox{smooth quartic}\\
\hline
6 & 4 & 6 &  2& 0&  0 & 2  \\   
%\mbox{quartic 1  node }  \\
 & 2 & 8 &  0& 0 & 1 & 3 \\
 %  \mbox{genus\ 3 \ hyperelliptic}, \mbox{elliptic fibration} \\
\hline
 5 & 5 & 7 &  2& 0 & 1 & 2 \\
 %&  & \mbox{quartic 1\ cusp},\  \mbox{hyp. 1 node} \\
  \hline
 4 & 6 & 8 &  2& 0& 2 & 2 \\
   %& & \mbox{hyp. 1 cusp} \\
  \end{array}
$$
\vspace{0.3cm}

\caption{The case $g>1$}\label{g>1}
\end{table}

\end{theorem}
\begin{proof}  
If $r=\rk S(\sigma)\geq 5$ then, by Proposition \ref{ellinv}, $X$ carries a $\sigma$-invariant elliptic fibration $\pi$. 
If $C\subset \Fix(\sigma)$ has genus $g>1$, then $C$ is transversal to the fibers of $\pi$, so that any fiber of $\pi$ is preserved by $\sigma$ and we are in the hypotheses of Corollary \ref{cor}. 
Thus $g=n=2$, $k=0$ and $a\in\{1,2\}$ by Corollary \ref{cor}  
(the case $g=3, n=k=0$, $a=1$ has $r=2$ and the case $g=n=2, k=a=0$ has $r=4$). 

We now assume that $r<5$.  
Since $l+2m\geq 18$, $l+m=10-2\alpha\leq 14$ by Proposition \ref{rel4} 
and $l-m=2a$ by Lemma \ref{rel4a},
then $4\leq m \leq 7$ and $a\leq 3$.

By Proposition \ref{rel4} and Lemma \ref{rel4a}
 the vector $(r,k,g,a)$ either appears in the rows of Table \ref{g>1} or is equal to one among 
 $(4,0,3,3)$, $(3,0,3,2), (4,1,3,0), (4,2,4,0)$.
 In any case we can compute $S(\sigma^2)$ (up to isomorphism) by means of the classification theorem of $2$-elementary lattices \cite[Theorem 4.3.1]{nikulinfactor}:
\begin{eqnarray*}
\begin{array}{c|c}
(r,k,g,a)&S(\sigma^2)\\
\hline
(4,0,3,3)&U\oplus E_8\oplus D_4\\
\hline
(3,0,3,2)& U\oplus D_8\oplus A_1^{\oplus 2}\\
\hline
(4,1,3,0)&U\oplus D_4\oplus A_1^{\oplus 4} \mbox{ or } U(2)\oplus D_4^{\oplus 2}\\
\hline
(4,2,4,0)&U\oplus D_4^{\oplus 2} \mbox{ or } U\oplus D_6\oplus A_1^{\oplus 2}.
\end{array}
\end{eqnarray*}
If $S(\sigma^2)\cong U\oplus R$, where $R$ is a direct sum of root lattices, then $g\leq 2$ by Theorem \ref{jac}, giving  a contradiction.
 In the case $(4,1,3,0)$ with $S(\sigma^2)\cong U(2)\oplus D_4^{\oplus 2}$ we still have an isomorphism $S(\sigma^2)\cong U\oplus R$ (where $R$ is not a sum of root lattices) by Remark \ref{split} 
 since $j=2$. By Theorem \ref{jac}, since $g>2$, there is a jacobian elliptic fibration $\pi$ whose fibers are invariant for $\sigma$ and with two sections $E,E'$ fixed by $\sigma^2$. Since $a=0$ these sections are fixed by $\sigma$, so that each fiber of $\pi$ contains four fixed points, giving a contradiction.
 %We are left with the case $(4,1,3,0)$ and $S(\sigma^2)\cong U(2)\oplus D_4^{\oplus 2}$. 
% By \cite[Lemma 2.1, 2.2]{kondoell} $X$ carries a $\sigma^2$- invariant elliptic fibration $\pi$ with no sections and two reducible fibers of type $I_0^*$:
%$2R+R_1+R_2+R_3+R_4$ and  $2R'+R_1'+R_2'+R_3'+R_4'$.
%Since $\sigma^2$ fixes the curve $C$ of genus $3$, then any fiber of $\pi$ is preserved by $\sigma^2$.
%Moreover, since $\sigma^2=\id$ on $\Pic(X)$, then each smooth rational curve is $\sigma^2$-invariant. This implies that $\sigma^2$ fixes $R$ and $R'$ since each of them contains $4$ fixed points. 
%Since $k=1$, one of these two curves is also fixed by $\sigma$, we can assume it to be $R$. 
%The curve $C$ meets the generic fiber in $4$ points and it intersects all the $R_i$'s and the $R_i'$'s. 
%This implies that the fibration is not $\sigma$-invariant, since otherwise the generic fiber should contain only $2$ fixed points for $\sigma$.
%Thus we can assume that $\sigma(R_1)\not=R_1$ and we have that $\beta:=R_1\cdot\sigma(R_1)\geq 2$ since the two curves at least intersect in $R_1\cap C$ and $R_1\cap R$. The sublattice of $S(\sigma)$ generated by the classes of $C,R$ and $R_1\cup \sigma(R_1)$ has the following intersection matrix
%\begin{eqnarray*}
%M:=\left(\begin{array}{ccc}
%4&0&2\\
%0&-2&2\\
%2&2&-4+2\beta
%\end{array}
%\right).
%\end{eqnarray*}
%Since $\det(M)=-16\beta+24<0$, we get a contradiction with the fact that $S(\sigma)$ has hyperbolic signature, thus this case does not appear.
\end{proof}

\begin{example}[\textbf{plane quartics}]\label{quartic} This construction is due to Kond\=o \cite{K1}. Let $C$ be a smooth plane quartic, defined as the zero set of a homogeneous polynomial $f_4\in \IC[x_0,x_1,x_2]$ of degree four. The fourfold cover of $\IP^2$ branched along $C$ is a K3 surface with equation $$t^4=f_4(x_0,x_1,x_2).$$ 
The covering automorphism 
$$\sigma(x_0,x_1,x_2,t)= (x_0,x_1,x_2, it)$$ is a non-symplectic automorphism of order four whose fixed locus is the plane section $t=0$, which is isomorphic to the curve $C$. Thus we have $a=n=k=0$ and $g=3$.
In case $C$ has  nodes or cusps, 
%with local equations: $$x^2+y^2=0\ (\mbox{node}),\quad x^2+y^3=0\ (\mbox{cusp}),$$
then the fourfold cover $X$ of $\IP^2$ branched along $C$ has rational double points of type $A_3$ and $E_6$ at the inverse images of a node and of a cusp of $C$ respectively.
The minimal resolution $\tilde X$ of $X$ is a K3 surface and the covering automorphism of $X$ lifts to a non-symplectic automorphism of $\tilde X$. 
If $C$ has a node, then the central component of the exceptional divisor of type $A_3$ is fixed by $\sigma^2$ and contains two fixed points for $\sigma$.
If $C$ has a cusp, then the exceptional curve is of type $E_6$, $\sigma^2$ fixes its two simple components and $\sigma$ exchanges them.
Thus, if $C$ is irreducible with $x$ nodes and $y$ cusps, then the invariants of $\sigma$ are $g=3-x-y$, $a=y$, $n=2(x+y)$ and $k=0$. Taking $(x,y)=(0,0), (1,0)$ and $(0,1)$ we obtain examples for the first, second and fourth case in Table \ref{g>1}. 
If $(x,y)=(2,0), (1,1)$ or $(0,2)$ we obtain examples with $g=1$ corresponding to the cases in Table \ref{g=1} with a fiber $C'$ of type ${\rm I}_4, {\rm I}_8$ and ${\rm I}_{12}$ respectively.
If $C$ is the union of a cubic and a line we obtain the case  in Table \ref{g=1} with a fiber $C'$ of type ${\rm IV^*}$ and $(k,n,a)=(1,6,0)$.
In \cite[Proposition 1.7]{a} the lattice $S(\sigma^2)$ has been computed in case $C$ is irreducible and generic with $x$ nodes and $y$ cusps.
%$$\begin{array}{lllll}
%(n,c)&   S(\sigma^2)\\
%\midrule
%\vspace{0.1cm}
%(0,0)& \langle 2\rangle\oplus A_1^{\oplus 7}\\
%(1,0)& U\oplus A_1^{\oplus 8}\\
%(2,0)& U\oplus A_1^{\oplus 6}\oplus D_4\\
%(3,0)& U\oplus A_1^{\oplus 6}\oplus D_6\\
%(1,1)& U\oplus A_1^{\oplus 2}\oplus D_4\oplus D_6\\
%(1,2)& U\oplus A_1^{\oplus 2}\oplus D_6\oplus E_8\\
%(2,1)& U\oplus A_1^{\oplus 2}\oplus D_4\oplus D_8\\
%(0,1)& U\oplus A_1^{\oplus 4}\oplus D_6\\
%(0,2)& U\oplus A_1^{\oplus 2}\oplus D_4\oplus E_8\\
%(0,3)& U\oplus A_1^{\oplus 2}\oplus E_8^{\oplus 2}\\
%\end{array}$$
\end{example}
\begin{example}[\textbf{hyperelliptic genus three curves}]\label{hyp} Let $\mathbb F_4$ be a Hirzebruch surface and $e,f \in \Pic(\mathbb F_4)$ be the class of the rational curve $E$ with $E^2=-4$ and the class of a fiber respectively.
A smooth curve $C$ with class $2e+8f$ is a hyperelliptic genus three curve.
Let $Y$ be the double cover of  $\mathbb F_4$ branched along $C$ and $X$ be the double cover of $Y$ branched along $C\cup R_1\cup R_2$, where $R_1\cup R_2$ is the inverse image of the curve $E$.
The surface $X$ is a K3 surface (see \cite{K2} and \cite[\S 3]{a}) with a non-symplectic automorphism $\sigma$ of order $4$ whose fixed locus is the inverse image of the curve $C$ and such that $\sigma(R_1)=R_2$. Observe that $\sigma^2$ fixes $C\cup R_1\cup R_2$.
In this case we have $g=3$, $n=k=0$ and $a=1$.

As in the previous example, if $C$ has at most nodes and cusps, then the minimal resolution $\tilde X$ of $X$ is again a K3 surface with a non-symplectic automorphism of order $4$. If $C$ is irreducible with $x$ nodes and $y$ cusps, then the invariants of $\sigma$ are $g=3-x-y$, $a=y+1$, $n=2(x+y)$, $k=0$. 
Taking $C$ with a node and a cusp we obtain examples for the fourth and the last case respectively in Table \ref{g>1}.
If $(x,y)=(2,0), (1,1)$ or $(0,2)$ we obtain examples with $g=1$ corresponding to the cases in Table \ref{g=1} with a fiber $C'$ of type ${\rm I}_8$, ${\rm I}_{12}$ and ${\rm I}_{16}$ respectively.
In \cite[\S4.9]{K2} the lattice $S(\sigma^2)$ has been computed in case $C$ is irreducible and generic with $x$ nodes and $y$ cusps.
%$$\begin{array}{lllll}
%(n,c)&  \mbox{Picard lattice}\\
%\midrule
%\vspace{0.1cm}
%(0,0)&  U(2)\oplus D_4^{\oplus 2}\\
%(1,0)& U\oplus D_4^{\oplus 2}\oplus  A_1^{\oplus 2}\\
%(2,0)& U\oplus D_4^{\oplus 2}\oplus D_6\oplus  A_1^{\oplus 2}\\
%(3,0)& U\oplus D_6^{\oplus}\oplus A_1^{\oplus 2}\\
%(4,0)& U\oplus D_8\oplus D_8\\
%(1,1)& U\oplus A_1^{\oplus 2}\oplus D_4\oplus E_8\\
%(1,2)& U\oplus  U\oplus E_8 \oplus D_{10}\\
%(2,1)& U\oplus A_1^{\oplus 2}\oplus D_6\oplus E_8\\
%(0,1)& U\oplus  D_8\oplus D_4\\
%(0,2)& U\oplus D_8\oplus E_8\\
%\end{array}$$
 \end{example}

 \section{$\Fix(\sigma^2)$ only contains rational curves}
 In this section we assume that the curves fixed by $\sigma^2$ are rational and that at least one of them is fixed by $\sigma$.  
 \begin{theorem}\label{rational}
Let $X$ be a K3 surface and $\sigma$ be a purely non-symplectic automorphism of order four on it.
If $\Fix(\sigma)$ contains a smooth rational curve and all curves fixed by $\sigma^2$ are rational, 
then the invariants associated to $\sigma$ are as in Table \ref{g=0}. All cases in the table do exist, see Example \ref{ex0}.\end{theorem}

\begin{table}[ht] 
$$ 
\begin{array}{ccc|ccccccl}
m & r & l &  n & k & a  \\
\hline
 4 & 10 & 4 &  6 & 1 & 0   \\
 %& &      \mbox{quartic 3 nodes } (k=0),\ \mbox{cubic+line } (k=1)  \\
  \hline
 3 & 13 & 3 &  8 &2 & 0 \\
 %& &  \mbox{2 conics},\ \mbox{line + cubic with 1 node}\\
  & 11 & 5 &  6& 1 & 1 \\
  %& &  \mbox{quar. 2 nodes+1cusp},\  \mbox{hyp. 3 nodes}\\
   \hline  
     2 & 16 & 2 &  10 &3 & 0 \\
     %& &  \mbox{conic+ 2 lines }\ \\
      & 14 & 4 &  8 &2 & 1 \\
      %& &  \mbox{line+cubic with 1 cusp} (k=1)\\
       & 12 & 6 &  6& 1 & 2 \\
       %& &    \mbox{hyp. 2 nodes + 1 cusp}\\
        \hline
          1 & 19 & 1 & 12 & 4 & 0  \\
          %& &  \mbox{4 lines}\\
      & 13 & 7 &  6 & 1 & 3  \\
 \end{array}
 $$
 \vspace{0.2cm}
 
 \caption{The case $g=0$}\label{g=0}
 \end{table}
 \begin{proof}
By Proposition \ref{rel4} and Lemma \ref{rel4a} we have that $l+m\leq 8$ and $l-m=2a$, 
thus the possible cases are those appearing in Table \ref{g=0} and $(r,k,a)=(17, 3,1), (15, 2,2)$.
In both cases $m=1$ and the surface $X$ is isomorphic to Vinberg's K3 surface. 
Moreover, by \cite[Lemma 1.5, (2)]{MO} or Remark \ref{split}, $X$ has a $\sigma$-invariant jacobian elliptic fibration,
since $\ell(S(\sigma))\leq \rk(S(\sigma)^{\perp})\leq 7\leq \rk S(\sigma) -3$.
By Example \ref{vinberg}, if $\sigma$ preserves an elliptic fibration on $X$, then $a\in \{0,3,4\}$. 
Thus the two cases do not appear.
  \end{proof}
 
 \begin{example}[\textbf{Vinberg's K3 surface}]\label{vinberg}
If $m=1$, then $S(\sigma^2)=\Pic(X)$ has maximal rank and $X$ is isomorphic to the unique K3 surface with 
 $$T(X)=T(\sigma^2)\cong \left( \begin{array}{cc} 2 & 0 \\
 0 & 2\end{array}\right),$$ 
 since it can be easily proved that, up to isometry, this is the only rank two, even, positive definite and $2$-elementary lattice. 
 The automorphism group of this K3 surface is known to be infinite and has  been computed by Vinberg in \cite[\S 2.4]{V}.
 In particular, it is known that $X$  is birationally isomorphic to the following quartics in $\IP^3$:
 $$x_0^4=-x_1x_2x_3(x_1+x_2+x_3),$$
  $$x_0^4=x_2^2x_3^2+x_3^2x_1^2+x_1^2x_2^2-2x_1x_2x_3(x_1+x_2+x_3),$$
 which are degree four covers of $\IP^2$ branched along the union of four lines in general position and an irreducible quartic with three cusps respectively.
The two covering automorphisms $x_0\mapsto ix_0$ induce non-conjugate order four non-symplectic automorphisms on $X$:
 the first one has $a=0$, $n=12$ and  fixes $4$ smooth rational curves (the proper transforms of the lines), the second one has $a=3$ (coming from the cusps), $n=6$ and fixes one smooth rational curve (the proper transform of the quartic). These give the last two examples in Table \ref{g=0}.
  
   All elliptic fibrations $\pi:X\map \IP^1$ are known to be jacobian by \cite[Theorem 2.3]{Ke} and have been classified by Nishiyama in \cite[Theorem 3.1]{Nishiyama}. We recall the classification in Table \ref{vinbergtable}, where $R$ is the lattice generated by the components of the reducible fibers not intersecting the zero section and $MW$ denotes the Mordell-Weil group of $\pi$.
 
 \begin{table}[ht] 
$$\begin{array}{cc|c|c}
\mbox{No.} & R &  MW & a\\
\hline
%\vspace{0.1cm}
1) & E_8^{\oplus 2}\oplus A_1^{\oplus 2} & 0 & 0,4\\
2) & E_8\oplus D_{10} & 0 & 0\\
3)& D_{16}\oplus A_1^{\oplus 2}& \IZ/2\IZ & 0,3 \\
4)& E_7^{\oplus 2}\oplus D_4 &  \IZ/2\IZ & 0,3,4 \\
5)& E_7\oplus D_{10}\oplus A_1 &  \IZ/2\IZ & --\\
6)& A_{17}\oplus A_1 &   \IZ/3\IZ & 0\\
7)& D_{18} & 0 & 0\\
8)& D_{12}\oplus D_6 &  \IZ/2\IZ& 0,3\\
9)& D_8^{\oplus 2}\oplus A_1^{\oplus 2}& ( \IZ/2\IZ)^2 & 0,3,4 \\
10)& A_{15}\oplus A_3 &  \IZ/4\IZ & 0,3,4\\
11)& E_6\oplus A_{11} & \IZ\oplus  \IZ/3\IZ& 0,3\\
12)& D_6^{\oplus 3} &  ( \IZ/2\IZ)^2& 3\\
13)& A_9^{\oplus 2} &  \IZ/5\IZ& 0
\end{array}$$
\vspace{0.2cm}

\caption{Elliptic fibrations of Vinberg's K3 surface}\label{vinbergtable}
 \end{table}
 
In each case we will compute the possible values taken by $a$. 
 We will apply the height formula \cite[Theorem 8.6]{shioda} and the notation therein. Moreover, we will apply Theorem \ref{jac} to determine the number of fixed curves of $\sigma^2$ contained in the reducible fibers.\\ 
  1) 
  %Since the fibration has $4$ reducible fibers then $\sigma$ acts as an involution on the basis of the fibration (it is not the identity since the $j$-invariant of the fiber $E_8$ is not equal to $1$).
  All curves fixed by $\sigma^2$ are contained in the reducible fibers.
  The automorphism $\sigma$ either preserves the ${\rm II^*}$ fibers or it exchanges them, giving 
  $a=0$ or $a=4$ respectively.\\
 2) 
 The reducible fibers contain exactly $8$ curves fixed by $\sigma^2$ and the two remaining fixed curves of $\sigma^2$ are transversal to the fibers, one of them is the unique section, the other is a $3$-section. This implies that both are $\sigma$-invariant and $\sigma$ preserves all components of ${\rm I}_6^*$, giving $a=0$.\\
 %In particular, as in the previous case, $\sigma$ acts as an involution on $\IP^1$. 
 %The automorphism $\sigma$ either preserves all components of $D_{10}$, giving $a=0$, or acts exchanging the two branches of the fiber. The second case can be excluded since the unique section of $\pi$ is $\sigma$-invariant.\\
 3)
 % As before, $\sigma$ acts with order two on $\IP^1$. 
 The reducible fibers contain exactly $7$ curves fixed by $\sigma^2$. The remaining fixed curves of $\sigma^2$ are transversal to the fibers and give two sections $s_0, s_1$ (assume $s_0$ to be the zero section) and a $2$-section. 
 The translation by the order two element in the Mordell-Weil group gives rise to a symplectic automorphism of order two. Since such automorphism has $8$ fixed points by \cite{Nikulin1}, then it  acts on  the fiber ${\rm I}_{14}^*$ as a reflection with respect to the central component, i.e. $s_0$ and $s_1$ intersect the fiber in simple components not meeting the same component.
 This implies that $\sigma$ acts on the components of ${\rm I}_{14}^*$ either as the identity  or as a reflection with respect to the central component, giving $a=0$ and $a=3$ respectively.\\
 4) 
 %As before, $\sigma^2$ acts as the identity on $\IP^1$. 
 As before, the reducible fibers only contain $7$ curves fixed by $\sigma^2$ and the fibration has two sections $s_0, s_1$ and a $2$-section fixed by $\sigma^2$. 
 If $\sigma$ preserves each fiber, then either $s_0, s_1$ are fixed by $\sigma$ or they are exchanged giving $a=0$ and $a=3$ respectively by Corollary \ref{cor}.
Otherwise, if $\sigma$ has order two on the basis of $\pi$, then it exchanges the two fibers of type ${\rm III^*}$, so that either $a=3$ or $a=4$. \\
 5) Since the fibration has three reducible fibers of distinct types, then $\sigma$ acts as the identity on $\IP^1$. This is not possible by Proposition \ref{ell}, thus this fibration can not be $\sigma$-invariant.\\
6) The fiber of type ${\rm I}_{18}$ contains $9$ curves fixed by $\sigma^2$ and is clearly $\sigma$-invariant.
%Moreover, the fibration has $3$ sections which are $\sigma^2$-invariant.
%This implies that $\sigma$ has order four on $\IP^1$ (otherwise $\sigma^2$ should fix $12$ curves).
Since $18\not\equiv 0\ (\!\!\!\mod 4)$, then $a=0$.\\
7) The fiber of type ${\rm I}_{14}^*$ contains $8$ curves fixed by $\sigma^2$, the remaining two curves fixed by $\sigma^2$ give a section and a $3$-section. Thus $a=0$.\\
%As before, we deduce that $\sigma$ acts with order two on $\IP^1$  and $a=0$.\\
8)  The reducible fibers contain $7$ curves fixed by $\sigma^2$, the remaining fixed curves of $\sigma^2$ give two sections $s_0,s_1$.
%As before $\sigma$ acts with order two on $\IP^1$ and it preserves the two reducible fibers.
By the height formula, $s_0$ and $s_1$ meet simple components intersecting distinct components in the ${\rm I}_{8}^*$ fiber and intersecting the same component in the ${\rm I}_2^*$ fiber. Thus either $a=0$ or $a=3$.\\
9) The reducible fibers contain $6$ curves fixed by $\sigma^2$ and the fibration has four sections $s_0,s_1, s_2, s_3$ fixed by $\sigma^2$. 
%This implies, as before, that $\sigma$ has order two on $\IP^1$.
Observe that the elliptic fibration has two fibers of type ${\rm I}_4^*$,  two of type ${\rm I}_2$ and no other singular fibers (since the Euler characteristic of $X$ is $24$).
Let $\Theta_i^1$, $\Theta_i^2$, $i=0,1,2,3$ denote the ``external'' components of the ${\rm I}_4^*$ fibers and $\Theta_i^3$, $\Theta_i^4$, $i=0,1$, denote the components of the ${\rm I}_2$ fibers, see also the notation of \cite[Theorem 8.6, p. 229]{shioda}.
By the height formula, we can assume that $s_1$ intersects $\Theta_1^1, \Theta_2^2, \Theta_1^3, \Theta_1^4$, $s_2$ intersects $\Theta_2^1, \Theta_1^2, \Theta_1^3, \Theta_1^4$ and $s_3$ intersects $\Theta_3^1, \Theta_3^2, \Theta_0^3, \Theta_0^4$.
Observe that $\sigma$ either preserves the fibers of type ${\rm I}_4^*$ and exchanges the fibers of type ${\rm I}_2$, or it exchanges both pairs of reducible fibers, or it exchanges the fibers of type ${\rm I}_4^*$ and it preserves those of type ${\rm I}_2$.
In the first case $a=0,3$ or $4$, according to the action of $\sigma$ on the sections $s_i$'s.
In the second case $a=4$ (observe that the $4$ fixed points of $\sigma$ are contained in two smooth fibers).
The last case does not appear since otherwise $\sigma$ should preserve all components of the fibers ${\rm I}_2$, which gives a contradiction since the components intersecting $s_1, s_2$ contain no fixed points.\\
10) In this case all fixed curves by $\sigma^2$ are contained in the two reducible fibers, thus $\sigma$ has order $4$ on $\IP^1$. Observe that the fibration has four sections $s_0,s_1,s_2,s_3$ preserved by $\sigma^2$, so that each of them intersects the curves fixed by $\sigma^2$.
This remark and the height formula imply that we can assume that $s_1$ intersects $\Theta_8^1, \Theta_0^2$, $s_2$ intersects $\Theta_4^1, \Theta_2^2$ and $s_3$ intersects $\Theta_{12}^1, \Theta_2^2$ (again we use the notation of \cite[Theorem 8.6, p. 229]{shioda}).
This implies that $\sigma$  either preserves all components of the reducible fibers ($a=0$), or it acts
on the sections as a permutation $(s_2 s_3)$ or $(s_0 s_1)$ ($a=3$), or as a permutation of type $(s_0 s_1)(s_2 s_3)$ ($a=4$).\\
11) In this case all fixed curves by $\sigma^2$ are contained in the two reducible fibers, thus $\sigma$ has order $4$ on $\IP^1$.
By the height formula we can assume that $s_1$ intersects $\Theta_1^1, \Theta_4^2$ and $s_2$ intersects $\Theta_2^1, \Theta_8^2$.
This implies that $\sigma$  either preserves all components of reducible fibers ($a=0$), or it acts
on the sections as a transposition ($a=3$).\\
12) The reducible fibers contain $6$  curves fixed by $\sigma^2$, the remaining fixed curves give four sections $s_i, i=0,1,2,3$. Observe that $\sigma$ has order two on $\IP^1$, it exchanges two fibers of type ${\rm I}_2^*$ and it preserves the third one.
By the height formula we can assume that $s_1$ intersects $\Theta_1^1, \Theta_3^2, \Theta_3^3$, $s_2$ intersects $\Theta_3^1, \Theta_3^2, \Theta_1^3$ and $s_3$ intersects $\Theta_3^1, \Theta_1^2, \Theta_2^3$. This implies that $\sigma$ acts on the sections as a transposition, so that $a=3$.\\
13) All fixed curves by $\sigma^2$ are contained in the two reducible fibers and
%thus $\sigma$ has order $4$ on $\IP^1$.
$\sigma$ preserves the two fibers of type ${\rm I}_{10}$ (since otherwise $a=5$, which is  not possible), so that $a=0$. \end{example}
 \begin{example}\label{ex0}
 Let $\tilde X$ and $C$ be as in Example \ref{quartic}.
 Taking $C$ with 3 nodes or with 2 nodes and a cusp we obtain examples for the first and third case in Table \ref{g=0} respectively. If $C$ is the union of two conics (or the union of a line and a nodal cubic), the union of a conic and two lines or the union of a line and a cuspidal cubic we obtain examples for the second, fourth and fifth case in the table respectively.
 Finally, as mentioned in the previous example, if $C$ is the union of four lines or has 3 cusps, we obtain  the last two cases in the table.
 Similarly, if $\tilde X$ and $C$ are as in Example \ref{hyp} and $C$ has 3 nodes, 2 nodes and a cusp, or 1 node and two cusps, we obtain examples for the third, sixth and the last case in Table \ref{g=0} respectively.
 \end{example}

\section{The case $l=0$}\label{trivpic}
In this section we will assume that $l=0$, i.e. $\sigma^*$ acts as the identity on $S(\sigma^2)$.

\begin{pro}\label{ell0}
Let $\sigma$ be a non-symplectic automorphism on a K3 surface $X$ such that $l=0$. 
Then $\Fix(\sigma)$ is the disjoint union of smooth rational curves and points, $r=2\,(\!\!\!\!\mod 4)$  and $a=0$.
\end{pro}
\begin{proof} 
By Theorem \ref{g1}, if $\Fix(\sigma)$ contains an elliptic curve, then $l\geq 4$. On the other hand, if $\Fix(\sigma)$ contains a curve of genus $g>1$, then $l=2a+m>0$ by Lemma \ref{rel4a} .
Thus the fixed locus of $\sigma$ only contains smooth rational curves.
%Assume that $\sigma$ fixes a curve of genus $g\geq 1$, then we get a contradiction by Theorem \ref{g1} and by Lemma \ref{rel4a} (since it says that $-m=2a$).
Observe that $r=2\,(\!\!\!\!\mod 4)$  by Proposition \ref{rel4}.
 If $a$ is not zero, then there are two rational curves $F_1, F'_1$ fixed by $\sigma^2$ such that $\sigma(F_1)=F'_1$.
If $f=[F_1]-[F'_1]$ in $\Pic(X)$, then $\sigma^*(f)=-f$ and $f\not=0$, contradicting $l=0$.
\end{proof}
\begin{lemma}\label{iso}
Let $L$ be a lattice which is the direct sum of lattices isomorphic to $U\oplus U$, $U\oplus U(2)$, $E_8$ or $D_{4k}$, $k\geq 1$. Then $L$ has an isometry $\tau$ with $\tau^2=-\id$ acting trivially on $A_L={\rm Hom}(L,\IZ)/L$.
\end{lemma}
\begin{proof}
It is known that the Weyl group of a lattice $L$ isomorphic to either $E_8$ or $D_{4k}$, $k\geq 1$,  contains an isometry $\tau$ with $\tau^2=-\id$ acting trivially on $A_L$ \cite{cc}.
An isometry $\tau$ of $U\oplus U$ as in the statement can be defined as follows:
$$\tau:\ e_1\mapsto e_2,\ e_2\mapsto -e_1,\ e_3\mapsto e_4,\ e_4\mapsto -e_3,$$
where $e_1,e_2$ and $e_3,e_4$ are the natural generators of the first and second copy of $U$.
Such an action can be defined similarly on $U\oplus U(2)$.
\end{proof}

By Proposition \ref{ell0} the fixed locus of $\sigma$ and of its square are as follows:
$$\Fix(\sigma^2)=C\cup (E_1\cup\cdots\cup E_k)\cup (G_1\cup \cdots\cup G_{n_1/2}),$$
$$\Fix(\sigma)=E_1\cup\cdots\cup E_k\cup\{p_1,\dots,p_n\},$$
where $C$ is a curve of genus $g$, $E_i, G_i$ are smooth rational curves such that
$\sigma(G_i)=G_i$ and each $G_i$ contains exactly two isolated fixed points of $\sigma$.
We will denote by $n_2=n-n_1$ the number of isolated fixed points of $\sigma$ contained in $C$.

\begin{theorem}\label{teotaki}
Let $\sigma$ be a non-symplectic automorphism on a K3 surface $X$ such that  $S(\sigma^2)=S(\sigma)=\Pic(X)$. Then the invariants of the fixed locus of $\sigma$ and the lattices $S(\sigma^2)$ and $T(\sigma^2)$ (up to isomorphism) appear Table \ref{l=0}. Moreover, all cases in the table do exist.

\end{theorem}
\begin{proof}
By Proposition \ref{rel4} we have that
$$\mathcal X(\Fix(\sigma^2))-\mathcal X(\Fix(\sigma))=2-2g-n_2=-2m.$$
 Moreover, since $\sigma$ acts on $C$ as an involution with $n_2$ fixed points, then $g\geq \frac{n_2}{2}-1$ by Riemann-Hurwitz formula.
These remarks, together with Proposition \ref{rel4} and the classification theorem of $2$-elementary even lattices \cite[Theorem 3.6.2]{nikulinintegral} give the list of cases appearing in the table and some more with $S(\sigma^2)$ isomorphic to one of the following lattices:
$$U\oplus A_1^{\oplus 4},\ U\oplus D_6\oplus A_1^{\oplus 2},\  U\oplus D_4\oplus A_1^{\oplus 4},\ U\oplus E_8\oplus A_1^{\oplus 4},\ U\oplus E_8\oplus E_7\oplus A_1,$$ 
$$U\oplus E_7\oplus A_1,\ (2)\oplus A_1.$$

 In the first five cases $X$ has a $\sigma$-invariant jacobian elliptic fibration $\pi:X\map \IP^1$ with more than two reducible fibers by Theorem \ref{jac}. Since $\sigma$ fixes $>2$ points in the basis of $\pi$, then it preserves each fiber of $\pi$. 
 This gives a contradiction since $\sigma$ should have two fixed points in each fiber while the fibration has a unique section fixed by $\sigma$.

We now show that the case $S(\sigma^2)\cong(2)\oplus A_1$ does not appear. Let $e,f$ be the generators of $S(\sigma^2)$ with $e^2=2, f^2=-2, e\cdot f=0$.  The class $e$ is nef and the associated morphism is a degree two map $\pi:X\map \IP^2$ which is the minimal resolution of the double cover branched along an irreducible plane sextic $S$ with a node at the image of the curve with class $\pm f$.
Since $\sigma^*(e)=e$, then $\sigma$ induces a projectivity $\bar \sigma$ of $\IP^2$ with $\bar\sigma^2=\id$ (since it fixes $\pi(C)$).
This implies that, up to a choice of coordinates, a birational model of $X$ and $\sigma$ are given as follows:
$$X:\ w^2=f_6(x_0,x_1,x_2),\quad \sigma(x_0,x_1,x_2,w)=(-x_0,x_1,x_2, iw),$$
where $f_6$ is a homogeneous degree six polynomial with $\sigma(f_6)=-f_6$ and singular at one point.
Observe that such polynomial $f_6$ contains $x_0=0$ as a component, giving a contradiction. 
If $S(\sigma)\cong U\oplus E_7\oplus A_1\cong (2)\oplus A_1\oplus E_8$, 
 then $X$ is the minimal resolution of the double cover of $\IP^2$ branched along an irreducible sextic with a node and a triple point of type $E_8$.  We can  exclude this case by an argument similar to the previous one.

If $T(\sigma^2)$ is any lattice appearing in Table \ref{l=0}, then it carries an isometry $\tau$ with $\tau^2=-\id$ and acting trivially on the discriminant group by Lemma \ref{iso}.
It follows that the isometries $\id_{S(\sigma^2)}$ and $\tau$ glue to give an order four isometry $\rho$ of $L_{K3}$.
By the Torelli-type Theorem \cite[Theorem 3.10]{namikawa} there exists a K3 surface $X$ with a non-symplectic automorphism $\sigma$ of order four such that $\sigma^*=\rho$ up to conjugacy.
Thus all cases in the table do exist.
\end{proof}
\begin{table}[h!] 
$$
\begin{array}{cc|cccc|ll}
m & r & n_1 &n_2 & k & g & S(\sigma^2) & T(\sigma^2) \\
\hline
10 & 2 &  2 & 2 & 0 & 10 & U	& U\oplus U\oplus E_8^{\oplus 2}\\ 
  & 2  & 0 & 4 & 0  & 9 & U(2)  & U\oplus U(2)\oplus E_8^{\oplus 2}\\
 \hline
 8 & 6 &  2 & 4 & 1 & 7 & U\oplus D_4 & U \oplus U\oplus E_8\oplus D_4\\
  & 6 & 0 & 6 & 1 & 6 &  U(2)\oplus D_4 & U\oplus U(2)\oplus E_8\oplus D_4\\
\hline
6 & 10 &  6 & 2& 2& 6 & U\oplus E_8 & U \oplus U \oplus E_8\\
 & 10 & 4 & 4 & 2 & 5 & U(2)\oplus E_8 & U\oplus U(2)\oplus E_8\\
 & 10 & 2 & 6 & 2& 4 & U\oplus D_4^{\oplus 2} & U\oplus U\oplus D_4^{\oplus 2}\\
 & 10 & 0 & 8 & 2 & 3 & U(2)\oplus D_4^{\oplus 2} & U\oplus U(2)\oplus D_4^{\oplus 2}\\ 
\hline
4 & 14 & 6 & 4 & 3 & 3 & U\oplus E_8\oplus D_4 & U \oplus U \oplus D_4\\
 & 14 & 4 & 6 & 3 & 2 &  U(2)\oplus E_8\oplus D_4 & U\oplus U(2)\oplus D_4\\
\hline
2 & 18 & 10 & 2 & 4 & 2 & U\oplus E_8^{\oplus 2} & U\oplus U\\
 & 18 & 8 & 4 & 4 & 1 & U(2)\oplus E_8^{\oplus 2} & U\oplus U(2)\\
\end{array}
$$
\  \\
\ \\
\caption{The case $l=0$}\label{l=0}
\end{table}

\begin{remark}
In all cases of Table \ref{l=0} the lattices $S(\sigma^2)$ and $T(\sigma^2)$ are $2$-elementary even lattices with $x^2\in\IZ$ for all $x\in L^{\vee}={\rm Hom}(L,\IZ)$.
A lattice theoretical proof of this fact and an alternative proof of Theorem \ref{teotaki} are given by Taki in\cite[Proposition 2.4]{Taki}.
\end{remark}

\begin{example}\label{hirz}
Consider the elliptic fibration $\pi:X\map \IP^1$ in Weierstrass  form given by
$$y^2=x^3+a(t)x+b(t),$$
where $a(t)$ is an even polynomial and $b(t)$ is an odd polynomial in $t$.
Observe that it carries the order four automorphism
$$(x,y,t)\mapsto (-x, iy,-t).$$
For generic coefficients $X$ is a K3 surface, the fixed locus of $\sigma^2$ is the union of the curve of genus $10$ defined by $y=0$ and the section $x=z=0$ and $\sigma$ fixes four points on them (in the fibers over $t=0, \infty$).
It follows by Theorem \ref{inv}, Proposition \ref{rel4} and \cite[Theorem 4.3.1]{nikulinfactor} that $S(\sigma^2)\cong U$, $r=2$, $l=0$ and $m=10$.

A geometric construction of this family of K3 surfaces can be given as follows.
Let $Y$ be the Hirzebruch surface $\mathbb F_4$ and $e,f\in \Pic(Y)$ be the classes of the $(-4)$-curve and of a fiber respectively. 
Observe that  a section of $-2K_Y=4e+12f$ is the disjoint union of the $(-4)$-curve $E$ and a curve $C$ with class $3e+12f$ ($e$ is in the base  locus of  $4e+12f$). The generic such $C$ is smooth 
%By \cite[Chapter V, Corollary 2.18]{hartshorne} we can assume $C$ to be irreducible. 
and the double cover $p:X\map Y$ branched along $C\cup E$ is a K3 surface. We denote by $\tilde C$ and $\tilde E$ the pull-backs of $C$ and $E$ by $p$, observe that $g(\tilde C)=10$ and $g(\tilde E)=0$. 
We will consider the affine coordinates $t=u_1/u_2$ and $x=v_1/v_2$, where $u_1, u_2$ give a basis of $H^0(Y,f)$ and $v_1\in H^0(Y,e+4f)$, $v_2\in H^0(Y,e)$ are non-zero sections. 
Let $\iota\in {\rm Aut}(Y)$ be the involution on $Y$ given by $(t, x)\mapsto (-t,-x)$ and let $\varphi=0$ be the equation of $\tilde C$ in local coordinates. 
%in $Y-E-F$, where $F=\{u_2=0\}$. 
If $\varphi(-t,-x)=-\varphi(t,x)$, then $\varphi(t,x)=x^3+a(t)x+b(t)$ where $a$ is an even, degree $8$ and $b$ is an odd, degree $11$  polynomial in $t$.
In this case $\iota$ lifts to an order $4$ automorphism $\sigma$ on $X$, in local coordinates:
$$X:\ y^2=\varphi(t,x),\qquad \sigma(t,x,y)=(-t,-x,iy).$$
The ruling of $Y$ induces a jacobian elliptic fibration on $X$ having $\tilde C$ as a trisection and $\tilde E$ as a section. The Weierstrass equation of such fibration and the action of $\sigma$ on it are clearly the same as the ones given for $\pi$ at the beginning of this example.
 \end{example}

\begin{example}\label{quadric} 
Consider the involution $\iota: ((x_0,y_0), (x_1,y_1))\mapsto ((x_0,-y_0), (x_1,-y_1))$ of $\IP^1\times \IP^1$ and let $f$ be a bihomogeneous polynomial of degree $(4,4)$ such that $\iota(f)=-f$.
If $C=\{f=0\}$ is a smooth curve, then the double cover of $\IP^1\times \IP^1$ branched along $C$ 
$$X:\ w^2=f(x_0,x_1,y_0,y_1)$$ is a K3 surface and carries the order four automorphism:
$$\sigma:\ (((x_0,y_0), (x_1,y_1)), w)\mapsto  (((x_0,-y_0), (x_1,-y_1)), iw).$$
The fixed locus of $\sigma^2$ is the genus $9$ curve defined by $w=0$ and $\sigma$ fixes $4$ points on it.
It follows by Theorem \ref{inv}, Proposition \ref{rel4} and \cite[Theorem 4.3.1]{nikulinfactor} that $S(\sigma^2)\cong U(2)$ (it is the pull-back of the Picard group of the quadric), $r=2$, $l=0$ and $m=10$.
\end{example}

\section{$\Fix(\sigma)$ only contains isolated points}
Let $\sigma$ be a purely non-symplectic automorphism of order four on a K3 surface $X$  having only isolated fixed points. It follows from Proposition \ref{rel4}
 that $\Fix(\sigma)$ contains exactly four points $p_1,\dots,p_4$. Moreover, the fixed locus of $\sigma^2$ is as follows:
 $$\Fix(\sigma^2)=C\cup (F_1\cup F'_1)\cup \cdots\cup (F_a\cup F'_a)\cup G_1\cup\cdots\cup G_{n_1/2},$$
where each $G_i$ is a smooth rational curve which contains $2$ fixed points of $\sigma$ and $C$ is a smooth genus $g$ curve which contains the remaining 
$n_2:=4-n_1$ fixed points.

\begin{theorem}\label{teopunti}
Let $\sigma$ be a purely non-symplectic automorphism of order $4$ having only isolated fixed points on a K3 surface $X$. Then $\sigma$ fixes exactly $4$ points. Moreover, if $\Pic(X)=S(\sigma^2)$ and $l>0$, then the invariants of $\Fix(\sigma^2)$ and the lattice $S(\sigma^2)$ (up to isomorphism) appear in Table \ref{alpha=0}.
{\small
\begin{table}[ht]
$$
\begin{array}{cc|ccc|l}
m & r  & n_1   & g &a & S(\sigma^2)   \\
\hline
 9 & 3 & 2 &  8 & 0 & U\oplus A_1^{\oplus 2}  \\
  &  & 0 &   7 & 0 & U(2)\oplus A_1^{\oplus 2} \\   
 \hline 
   8 & 4     & 2  &  6 &  0 & U\oplus A_1^{\oplus 4} \\
    &    &   0  &  5 & 0& U(2)\oplus A_1^{\oplus 4}  \\
   \hline
  7 &  5     & 2  &  4 & 0 &U\oplus A_1^{\oplus 6}  \\
   &    &     0  &  3 &  0&U(2)\oplus A_1^{\oplus 6} \\
   \hline
6 & 6  &  4 &  3 & 0 & U(2)\oplus D_4^{\oplus 2}   \\
   &    &  2 &  2 & 0&U\oplus A_1^{\oplus 8} \\
   &    &  0 &  1 &  0 &U(2)\oplus A_1^{\oplus 8}, U\oplus E_8(2)\\
   &    &  0 &  3 &  1 & U\oplus D_4\oplus A_1^{\oplus 4} , U(2)\oplus D_4^{\oplus 2} \\
   &    &  2 &  4 &  1 &U\oplus D_4^{\oplus 2}, U\oplus D_6\oplus A_1^{\oplus 2} \\
   \hline
5 & 7  & 4  &  1 & 0  & U(2)\oplus D_4^{\oplus 2}\oplus  A_1^{\oplus 2} \\
  &    & 2  &  0 & 0 &U\oplus A_1^{\oplus 10} \\
   &        & 0  &  1 &  1   & U(2)\oplus D_4^{\oplus 2}\oplus  A_1^{\oplus 2}   \\ 
   &        & 2  &  2 &   1& U\oplus D_4^{\oplus 2}\oplus A_1^{\oplus 2}   \\   
   &        & 0  &  3 &  2  & U\oplus E_7\oplus A_1^{\oplus 3} \\
   &        &2   &  4 &  2 & U\oplus E_8\oplus A_1^{\oplus 2}\\
%\end{array}
%\begin{array}{cc|cccl}
%m & r  & n_1   & g &a   \\
\hline
4 & 8  & 2  &  0 & 1 & U\oplus D_4^{\oplus 2}\oplus A_1^{\oplus 4}  \\
   &        & 4  &  1 &   1& U(2)\oplus D_6^{\oplus 2}, U\oplus D_4^{\oplus 3}\\
   &        & 0    & 1 &  2  &U(2)\oplus D_6^{\oplus 2}, U\oplus D_4^{\oplus 3}\\ 
   &        & 2    &  2 &  2 & U\oplus D_6^{\oplus 2}, U(2)\oplus D_4\oplus E_8\\
   &        &0       &  3 &  3  & U\oplus E_8\oplus D_4 \\
\hline
3 & 9  & 2  &  0 & 2  & U\oplus D_6^{\oplus 2}\oplus A_1^{\oplus 2}  \\
 &          & 4  &  1 &  2 & U(2)\oplus E_7^{\oplus 2}  \\
  &         & 0  &  1 &   3  &  U(2)\oplus E_7^{\oplus 2}\\
   &          & 2  &  2 & 3 & U\oplus E_7^{\oplus 2}  \\
\hline
2 & 10  & 2  &  0 & 3  &U\oplus E_7^{\oplus 2}\oplus A_1^{\oplus 2}, U\oplus D_8^{\oplus 2} \\  
  &           & 4  &  1 &  3   & U\oplus E_8\oplus E_7\oplus A_1,U(2)\oplus E_8^{\oplus 2} \\   
  &           & 0  &  1 &   4  & U\oplus E_8\oplus E_7\oplus A_1,U(2)\oplus E_8^{\oplus 2}  \\   
\hline
1 & 11  & 2  &  0 &  4  &U\oplus E_8^{\oplus 2}\oplus A_1^{\oplus 2} 
  \end{array}
  $$
  \ \\
  \ \\
\caption{The case $\alpha=k=0,\ l>0$}\label{alpha=0}
\end{table}
}
\end{theorem}
\begin{proof}    Since $\alpha=k=0$, it follows from Proposition \ref{rel4} that $n=4$, $r=l+2$ and $m=10-l$. The fixed locus of $\sigma^2$ contains $j=2a+n_1/2$ smooth rational curves and a curve of genus $g$. Thus by Theorem \ref{inv} we obtain $r+l=11-g+2a+n_1/2$.

Observe that the case $l=1, n_1=4$ does not exist since in this case, by Theorem \ref{inv} and \cite[Theorem 4.3.1]{nikulinfactor} (see Figure 1 in \cite{ast}), a curve fixed by $\sigma^2$ has genus $\leq 8$. 
 
 We now assume that $S(\sigma^2)\cong U\oplus R$, where $R$ is a direct sum of root lattices. 
 If $g>4$, then by Theorem  \ref{jac} the surface $X$ carries a $\sigma$-invariant jacobian elliptic fibration $\pi$ with reducible fibers of type $R$.
The automorphism $\sigma$ acts as an involution on $\IP^1$ and preserves two fibers $F,F'$ of $
\pi$. The curve $C\subset \Fix(\sigma^2)$ of genus $g$ intersects each fiber at $3$ points and  $C\cap F$, $C\cap F'$ are $\sigma$-invariant. This implies that $C$ contains at least two fixed points. Moreover, $\pi$ has a unique section fixed by $\sigma^2$ which is $\sigma$-invariant. Thus $n_1=n_2=2$. 
In particular $a=0$ if the rational curves fixed by $\sigma^2$ are at most $2$.

Observe that, since $C$ intersects the generic fiber of $\pi$ in three points, then it is trigonal.
If $g\geq 3$, this implies that $C$ is not hyperelliptic, hence the canonical morphism of $C$ is an embedding in $\IP^{g-1}$. 
The involution $\sigma$ on $C$ is thus induced by an automorphism of the projective space.
If $g=3$, this implies that $C$ is isomorphic to a plane quartic and, since an involution of $\IP^2$ fixes a line, then $n_2>0$. 
 
 If $(m,r,g,a)=(8,4,6,0)$ and $S(\sigma^2)\cong U(2)\oplus D_4$, then $X$ is the double cover of $\IP^2$ branched along the union of a smooth quintic and a line \cite[\S 4.2]{L}.
 By \cite[Corollary 2.4.3]{FP} an involution on a smooth plane quintic has quotient of genus two, equivalently  it has $6$ fixed points. Thus this case does not appear.

If $(m,r,g,a)=(7,5,6,1)$, then $S(\sigma^2)\cong U\oplus D_6$. Observe that $\sigma^2$ fixes the section of $\pi$ and two irreducible components of the fiber of type ${\rm I}_2^*$. Since $a=1$ and the fibration is $\sigma$-invariant, then $\sigma$ should act as a reflection on the fiber ${\rm I}_2^*$, but this is not possible since the fibration has a unique section. Thus this case does not appear.

The cases in the table are then obtained by taking all possible values for $l$, using the equations and remarks at the beginning of the proof. The lattices can be computed by means of  Theorem \ref{inv}, Proposition \ref{rel4}, and the classification theorem of $2$-elementary lattices \cite[Theorem 4.3.1]{nikulinfactor}.
\end{proof}

\begin{example}\label{hirz1}
Let $C, E\subset \mathbb F_4$ as in Example \ref{hirz}. If $C$ has rational double points, then the minimal resolution $X$ of the double cover of $\mathbb F_4$ branched along $C\cup E$ is still a K3 surface.
If $C$ has two nodes exchanged by $\iota$, then $\iota$ lifts to an order four automorphism $\sigma$ of $X$ such that $\sigma^2$ fixes a curve of genus $8$ (the pull-back of the proper transform of $C$) and a smooth rational curve $\tilde E$ (the pull-back of the $(-4)$-curve $E$).
Observe that here $l>0$ since the exceptional divisors over the two nodes are exchanged by $\sigma$.
The lattice $S(\sigma^2)$ in this case is isomorphic to $U\oplus A_1^{\oplus 2}$.

If $C$ has two triple points exchanged by $\iota$, then $\sigma^2$ fixes the pull-back of the proper transform of the curve $C$,  $\tilde E$ and the central components of the resolution trees over the two singular points, which are of type $\tilde D_4$. In this case $a=1$, since such components are exchanged by $\sigma$, and the lattice $S(\sigma^2)$  is isomorphic to $U\oplus D_4^{\oplus 2}$.
Similarly, we get examples if the triple points of $C$ are simple singularities of type $D_n (n\geq 6), E_7, E_8$.

Considering $C$ with ordinary nodes (up to $10$) and triple points exchanged by $\iota$ we obtain several examples for the cases in the table with $S(\sigma^2)\cong U\oplus R$, where $R$ is a direct sum of root lattices.

Similarly, we can construct examples for the cases of type $U(2)\oplus R$ by  generalizing Example \ref{quadric} to the case when the curve $C$ in $\mathbb P^1\times \mathbb P^1$ has simple singularities. 
%In this way, we obtain examples for the cases in Table \ref{alpha=0}, excepted the ones in gray color.
\end{example}

\begin{example}\label{jaco} Consider the jacobian elliptic fibration $\pi:X\map \IP^1$ defined as follows:
$$y^2=x(x^2+a(t^4)x+b(t^4)),$$
where $a, b$ are polynomials of degree $1$ and $2$ respectively.
Observe that $\pi$ has a $2$-torsion section $t\mapsto (0,0,t)$. The translation by this section gives a symplectic involution $\tau$ on $X$.
Moreover, $X$ has the order four non-symplectic automorphisms $\sigma: (x,y,z,t)\mapsto (x,y,z,it)$ and $\sigma':=\sigma\circ \tau$.
For generic $a,b$, $\pi$ has $8$ fibers of type ${\rm I}_2$ and $8$ fibers of type ${\rm I}_1$.
The automorphisms $\sigma$ and $\sigma'$ act with order four on $\IP^1$, preserve the two smooth fibers over $t=0$ and $t=\infty$ and act as an involution over $t=\infty$. Moreover, $\sigma$ fixes pointwisely the fiber over $t=0$, while $\sigma'$ acts as an order two translation on it.

For special choices of $a$ and $b$ we can obtain examples with reducible fibers of type ${\rm I}_{4M}$ over $t=0$ or $t=\infty$.
For example, if $a(t^4)=t^4$ and $b(t^4)=1$, then $\pi$ has a smooth fiber over $t=0$ and a fiber of type ${\rm I}_{16}$ over $t=\infty$.
The automorphism $\sigma$ fixes pointwisely the fiber over $t=0$ and acts on the fiber over $t=\infty$ as a reflection which leaves invariant the components $\Theta_0, \Theta_8$ (see the notation in Example \ref{vinberg}), giving 
$k=0, a=3, n=4$.
The symplectic involution $\tau$ acts over $t=\infty$ as a rotation sending $\Theta_0$ to $\Theta_8$.
Finally, the automorphism $\sigma'$ acts over $t=0$ as a translation and over $t=\infty$ as a reflection which leaves invariant the components $\Theta_4,\Theta_{12}$, giving $k=0, a=3, n_2=0, n_1=4$ (see Table \ref{alpha=0}). 
For more details on this example se also \cite[Proposition 4.7]{vGS}.
\end{example}

\section{The other cases}\label{other}
At this point of the classification the cases left out are those with
$$
\Fix(\sigma)=E_1\cup\cdots\cup E_k\cup\{p_1,\ldots,p_n\} 
$$
$$
\Fix(\sigma^2)=C\cup (E_1\cup\cdots\cup E_k)\cup(F_1\cup F_1'\cup\cdots\cup F_a\cup F_a')\cup (G_1\cup\cdots \cup G_{n_1/2})
$$
where $C$ is a curve of genus $g>0$, $k>0$, $l>0$ and $n_2=n-n_1$ is the number of fixed points on $C$. Observe that in this case $\alpha=k$. 
%so $n=2k+4$  
 %and $2k=10-l-m$ by Proposition \ref{rel4}. In particular, observe that $m+l$ is even and $m+l\leq 8$.

\begin{theorem}\label{curvefissate}
Assume that $g=g(C)>1$. Then $g\leq m$ and we are in one of the following cases:
$$
\begin{array}{c|c|c|c}
m+l&k&g\leq &a\leq  \\
\hline
4&3&3&2\\
6&2&5&3\\
8&1&7&4
\end{array}
$$

\end{theorem}
\bprf 
By computing the
difference $\chi(\Fix(\sigma^2))-\chi(\Fix(\sigma))$ topologically and using the Lefschetz's formula, one gets
the relation 
%\begin{eqnarray*}
%2-2g-n_2+4a=2l-2m 
%\end{eqnarray*}
%so that
\begin{equation}\label{punti}
g-2a=m-l+1-\frac{n_2}{2}.
\end{equation}
%Using Hurwitz formula on $C$ we obtain 
%\begin{eqnarray}\label{hurwitz}
%n_2\leq 2g+2. 
%\end{eqnarray}
By Theorem \ref{inv} we also have:
\[
g+j=g+2a+k+\frac{n_1}{2}\leq 11.
\]
%which gives
%\begin{eqnarray}\label{punti2}
%g+2a\leq 11-k-\frac{n_1}{2}.
%\end{eqnarray}
By the previous two relations and Proposition \ref{rel4} we get
\[
g\leq 5-k+\frac{m-l}{2}=m,
\]
and
\begin{equation}\label{a}
4a\leq 8-2k+n_2+l-m\leq 12+l-m,
\end{equation}
where we obtain the last inequality using $n_2\leq n=2k+4$. 
By Proposition \ref{rel4} we have that $m+l$ is even and $m+l\leq 8$,
so that $m+l=4,6,8$.
%This gives the cases in the table.
%This implies that $g\leq 7$ and $a\leq 4$. 
%By \eqref{genus} we have that $g\leq m$, in particular $m\geq 2$.
%Moreover, by the previous conditions we have that $m+l=4,6,8$. 
%Substituting 
%$$
%k=5-\frac{l+m}{2}
%$$
%By \eqref{genus} we get $g\leq m$ hence $m\geq 2$. This excludes the case $m+l=2$ and gives  
%This gives the inequalities for $g$ in the table. 
%If $m+l=4,6$ then $|m-l|\leq 2$  and  $|m-l|\leq 4$ respectively, so by \eqref{a} we get  $a\leq 3$ and $a\leq 4$ respectively.
%Similarly, if $m+l=8$ then $a\leq 4$.\\

We now show that the case $m+l=4$, $k=3$ and $a=3$ is not possible. By \eqref{a} we get $l\geq m$ and by \eqref{punti}
with $n_2\leq 2k+4=10$ we get $g\leq 2+m-l$. If $l>m$ then $g<2$, which is not possible. If $l=m$ then we  have 
$g=2$, hence again  $n_2=10$ by \eqref{punti}, contradicting the Riemann-Hurwitz formula. The case
$m+l=6$, $a=4$ can be excluded similarly.\\  
\qed

\begin{example}\label{wei2}
Consider the following elliptic K3 surface $\pi:X\map \IP^1$ in Weierstrass form:
$$
y^2=x^3-a(t)x,\ \ \deg a(t)=8.
$$
with the order four automorphism
$$\sigma(x,y,t)=(-x,iy,t).$$
The automorphism $\sigma$ fixes the two sections $x=y=0$, $x=z=0$ (hence $k\geq 2$), while $\sigma^2$ also fixes 
the curve $C:\ y=x^2-a(t)=0$.
The discriminant of the fibration is $\Delta(t)=4a^3(t)$, hence for a generic polynomial $a(t)$ the fibration has $8$ fibers of type ${\rm III}$ (more precisely these are of type ${\rm III }\, b)$). 
Moreover, the curve $C$ has genus $3$ and $\sigma$ has $8$ fixed points on it, so that: $k=2$, $n_2=8$, $a=0$. One can compute also that $l=0$, so that this case appears in Table \ref{l=0}.
We now study how $a(t)$ can split:
\begin{enumerate}[i)]
\item If $a(t)=a_1(t)^2a_6(t)$, then $\pi$ has a fiber of type ${\rm I_0^*}$ and $6$ of type ${\rm III }\, b)$. In this case the ramification points on $C$ are $6$, so that $g(C)=2$. Here $k=2$, $n_1=2$, $n_2=6$, $a=0$.
Thus we have an example in Theorem \ref{curvefissate}.
\item If $a(t)=a_1(t)^2b_1(t)^2a_4(t)$ then $\pi$ two fibers of type ${\rm I_0^*}$ and four of type ${\rm III}\, b)$. Here $k=2$, $n_1=4$, $n_2=4$, $a=0$ and $g=1$. This case appears in Theorem \ref{othergenus1}.
\item If $a(t)=a_1(t)^2b_1(t)^2c_1(t)^2a_2(t)$ then we have $3$ fibers of type ${\rm I_0^*}$ and two of type ${\rm III}\, b)$, so that $k=2$, $n_1=6$, $n_2=2$, $a=0$ and $g=0$. This case appears in Table \ref{g=0}.
\item If $a(t)=a_1(t)^2b_1(t)^2c_1(t)^2d_1(t)^2$ then we have four fibers of type ${\rm I_0^*}$. In this case $C$ splits into the union of two rational curves. In this case we have $k=2$, $n=8$, $a=1$. We are again in one case of Table \ref{g=0}.
\end{enumerate}
For generic $a(t)$ the Mordell-Weil group of $X$ is isomorphic to $\IZ/2\IZ$. The translation by a generator of  the group is a symplectic involution $\iota$ on $X$ and
the composition $\sigma'=\iota \circ \sigma$ is again a non-symplectic automorphism of order $4$ on $X$ which fixes $C$ and exchanges the two sections of $\pi$. 
%so that for generic $a(t)$: $g=3$, $k=0$, $n=0$, $a=1$. 
\end{example}

\begin{example}\label{singu} An easy computation with MAPLE 
 shows that there are 63 possible cases in Theorem \ref{curvefissate}.
One can produce some examples with the fixed locus described there
 starting from Example \ref{hirz} (resp. Example \ref{quadric}) and imposing simple 
singularities on the curve $C$ of genus 10 (resp. genus 9) such that at least one singular point
is invariant for $\iota$. 
For example, one can assume that the curve $C$ in Example \ref{hirz} has an ordinary triple point at an invariant point of $\iota$ and two nodes exchanged by $\iota$. The Dynkin diagram of the resolution over the triple point, which is of type $\tilde D_4$, is $\sigma$-invariant: its simple components are preserved and the double component is fixed by $\sigma$. Thus $k=1$, $n_2=4$, $n_1=2$, $a=0$ and $g=5$.
Observe that if $C$ has just an ordinary triple point at an invariant point, then we are in the case $g=7, k=1, n_1=2, n_2=4$ of Table \ref{l=0}.
Similar examples can be constructed from Example \ref{quadric}.
\end{example}

We now consider the case when $\sigma^2$ fixes an elliptic curve.
\begin{theorem}\label{othergenus1}
With the previous notation, if $g(C)=1$ then we are in one of the cases appearing in Table \ref{others}.
\begin{table}[ht]
$$
\begin{array}{ccc|cccc|c}
m & r & l & n_1& n_2 & k&a & \mbox{type of } C'\\
\hline
5&9&3&2&4&1&0 & {\rm I}_4\\
\hline
4&12&2&4&4&2&0 & {\rm I}_8\\
%&10&4&6&0&1&0 & \\
\hline
3&15&1&6&4&3&0& {\rm I}_{12}\\
%&13&3&8&0&2&0\\
\hline
%4&12&2&4&4&2&0& {\rm IV^*}\\
4&10&4&2&4&1&1&{\rm IV^*}\\
4&10&4&6&0&1&0&{\rm IV^*}\\
\end{array}
$$
\caption{The case $g=1, k>0, l>0$.}\label{others}
\end{table}
\end{theorem}

\bprf
Using the relations in the proof of Theorem \ref{curvefissate}  one can find the values in the table, we now show that these are the only possibilities.\\
Since $\sigma$ preserves $C$, then there is a $\sigma$-invariant elliptic fibration $\pi_C:X\map \IP^1$ with fiber $C$. Observe that  $\sigma$ has order four on the basis of $\pi_C$, since otherwise $\sigma^2$ would act as the identity on the tangent space at a point of $C$.
Thus $\sigma$ has two fixed points on $\IP^1$, corresponding to the fiber $C$ and a fiber $C'$ of $\pi_C$.
This implies that all rational curves fixed by $\sigma$ are contained in $C'$, so that $C'$ is reducible (since $k>0$).
Observe that $\sigma$ acts on $C$ either as an involution with four fixed points or as a translation.
By Proposition \ref{rel4} we have $n\geq 6$, so that $C'$ contains at least two fixed points of $\sigma$. 
This excludes the Kodaira types ${\rm I}_2,\,{\rm I}_3,\,{\rm III},\,{\rm IV}$ for $C'$.  
By similar arguments as in the proof of Theorem \ref{g1} also the types ${\rm I}_N^*$, ${\rm II^*}$ and ${\rm III^*}$ can be excluded. If $C'$ has Kodaira type ${\rm IV^*}$ then $\sigma$ either preserves each component or it exchanges two branches. In the first case $a=0$ and either $n=n_1=6$ with $k=1$ or $n_1=n_2=4$ with $k=2$. By Lemma \ref{jimmy} the last case is not possible since if each rational curve is $\sigma$-invariant, then ${\rm IV^*}$ contains only one fixed rational curve.  
In the second case $a=1$, $k=1$, $n_1=2$ and $n_2=4$.

We now consider the case when $C'$ is of type ${\rm I}_N$, $N\geq 4$. 
Since $C'$ contains at least a fixed curve for $\sigma$, then all
components of ${\rm I}_N$ are preserved by $\sigma$,  hence $a=0$. By the previous remarks, since $n_2=4$ or $0$, then either $n_1=2k$ or $n_1=2k+4$ respectively.

If $n_2=4$, then $N=2k+n_1=4k$. For $N>12$ we get $k>3$, which gives $m<3$ by Proposition \ref{rel4}. By the equality \eqref{punti} we get $m-l=2$ so $m=2$ and $l=0$, contradicting the assumption $l>0$. 
Hence we only have the cases in the table. 

On the other hand, if $n_2=0$ we get $N=2k+(2k+4)=4k+4$. By equation \eqref{punti} we get $l=m$ and $k+m=5$ by Proposition \ref{rel4}. If $N\geq 16$ then $k\geq 3$, which gives $m<3$ as before.
If $m=l=1$, then $k=4$, $n_1=12$ and $C'$ is of type ${\rm I}_{20}$. 
%Since $a=0$ the lattice $S(\sigma)$ contains the classes of the $20$ components of $C'$, contradicting $r=19$.  
If $m=l=2$, then $k=3$, $n_1=10$ and $C'$ is of type ${\rm I}_{16}$. 
%This contradicts Lemma \ref{jimmy} since a fiber ${\rm I}_{16}$ contains $4$ fixed rational curves. 
If $m=l=3$ (or $m=l=4$) then $k=1$ (respectively $k=2$) and we have a fiber ${\rm I}_8$ (respectively ${\rm I}_{12}$). 
All these cases can be excluded by Lemma \ref{jimmy} since if  $a=0$ then a fiber of type ${\rm I}_{4M}$, $M\geq 2$ has $k=M$. 
%Since $a=0$ the lattice $S(\sigma)$ contains the classes of the $16$ components of $C'$. This gives a contradiction since $r=16$ and $S(\sigma)$ is hyperbolic by Lemma \ref{lat}.  \qed

\begin{example}\label{ultimo} Observe that the elliptic K3 surface $\pi_C:X\map \IP^1$ in Example \ref{wei} also carries the non-symplectic order four automorphism
$$\sigma'(x,y,t)=(x,-y,it),$$ 
obtained by composing the automorphism $\sigma$ defined there with the non-symplectic involution $y\mapsto -y$.
The automorphism $\sigma^2=(\sigma')^2$ is an involution fixing the smooth elliptic curve over $t=0$ and some rational curves in the singular fiber over $t=\infty$. 
Generically, $\sigma'$ acts on the elliptic curve over $t=0$ as an involution with $4$ isolated fixed points and as the identity on the fiber over $t=\infty$.
If the fiber over $t=\infty$ is reducible, then $\sigma$ fixes at most rational curves, in particular $k=\alpha$.
Observe that $\sigma'$ preserves the section at infinity $t\mapsto (0:1:0;t)$ and has two fixed points on it, one on the fiber over $t=0$, the other over $t=\infty$. Using this condition and the equality $n=2\alpha+4$ given by Proposition \ref{rel4}, one sees that $n=6$, $k=1$ if $\pi_C$ has a fiber of type ${\rm I}_4$, $(k,n,a)=(2,8,0)$  for a fiber
of type  ${\rm I}_8$, $(3,10,0)$ for a fiber of type ${\rm I}_{12}$ and $(k,n,a)=(4,12,0)$ for a fiber of type  ${\rm I}_{16}$.
In case $\pi_C$ has a fiber of type ${\rm IV^*}$,  $\sigma'$ acts as a reflection fixing one curve, giving $(k,n,a)=(1,6,1)$.

The cases with the fibers of type ${\rm I_4, I_8, I_{12}}$ and ${\rm IV^*}$ give examples with $l\not=0$ that appear in Table \ref{others}.
 In case there is a fiber of type ${\rm I_{16}}$ 
 we have instead $l=0$ (see the last line of Table \ref{l=0}).
\end{example}

\bibliographystyle{amsplain}
\bibliography{Biblio}

\end{document}